\documentclass[11pt]{article}
\usepackage{amsthm}
\usepackage{amsmath}
\usepackage{amssymb}
\usepackage{bbm}
\usepackage{multirow}
\usepackage{mathrsfs}
\usepackage{graphicx}
\usepackage{enumerate}
\usepackage{bm}
\usepackage{verbatim}
\usepackage{natbib}
\usepackage{setspace}
\usepackage{subfigure}
\usepackage{threeparttable}
\usepackage{float}
\usepackage{epstopdf}
\usepackage[linesnumbered,ruled]{algorithm2e}
\usepackage{algorithmic}
\usepackage{booktabs}
\usepackage{dcolumn}

\newcolumntype{d}[1]{D{.}{.}{#1}}

\newtheorem{corollary}{Corollary}

\newtheorem{theorem}{Theorem}
\newtheorem{lemma}{Lemma}
\newtheorem{proposition}{Proposition}

\begin{document}
\onehalfspacing

\title{Chance Constrained Program with Quadratic Randomness: A Unified Approach Based on Gaussian Mixture Distribution}
\author{Xiaochuan Pang \thanks{School of Business, Sun Yat-Sen University,
Guangzhou 510275, China. Email: pangxch@mail2.sysu.edu.cn.} \quad Shushang Zhu \thanks{Corresponding author. School of Business, Sun Yat-Sen University, Guangzhou 510275, China. Email: zhuss@mail.sysu.edu.cn.} \quad Zhaolin Hu \thanks{School of Economics and Management, Tongji University, Shanghai 200092, China. Email: russell@tongji.edu.cn}
}

\maketitle

\begin{abstract}
This paper investigates the stochastic program with the chance constraint on a quadratic form of random variables following multivariate Gaussian mixture distribution (GMD). Under some mild conditions, it is proved that the asymptotic distribution of this kind of quadratic randomness is a univariate GMD. This finding helps to translate the chance constrained program into a more tractable one, based on which an effective branch-and-bound algorithm that takes advantage of the special structure of the problem is introduced to search the global optimal solution. Furthermore, it is shown that the error resulting from approximating the quadratic randomness with its associated asymptotic distribution can be reduced by restricting the condition numbers of covariance matrices of the multivariate GMD's components. In addition, some numerical simulations are also carried out to verify the effectiveness of this flexible and unified approach.
\end{abstract}
\vskip 0.3cm\hskip 0.3cm
\textit{Keywords:} Chance constrained program, Gaussian mixture distribution, Quadratic randomness, Asymptotic distribution, Branch-and-bound

\section{Introduction}
Chance constrained program (CCP) is a typical stochastic program for dealing with random uncertainty in decision making. Initial from \cite{charnes1958}, \cite{miller1965} and \cite{prekopa1970}, CCP has been widely applied in engineering, finance, and other fields. For instance, in finance, the widely encountered problem of maximizing the portfolio return subject to the constraint on value-at-risk (VaR) is actually a special case of CCP, where VaR is the risk measure characterized by a quantile of the loss distribution (see, e.g., \cite{jorion2007}). However, as \cite{nemirovski2006} point out, even with the linear structure on the chance constraint, the problem may still be non-convex. Thus CCP is generally computationally intractable, which hinders the applications of the model in practice.

Generally, CCP considers the probabilistic constraint on a randomness that is a function of some random variables and decision variables. Due to the difficulty in solving CCP, previous literature mainly considers the  randomness taking a linear form of the random variables. \cite{shapiro2009} discuss the conditions on the form of randomness and the distribution of random variables under which the CCP is a convex program. \cite{henrion2012} provide the gradient formula for linear chance constraints under a possibly singular multivariate Gaussian distribution. There is a branch of literature on linear distributionally robust CCP or its variant VaR optimization via defining the uncertain set of the distribution of random variables with different approaches, such as the known first- and second-moment information of \cite{el2003}, the radially symmetric distribution set of \cite{calafiore2006}, and the known expectation and dispersion function of \cite{hanasusanto2017}.

In this paper, we consider the general form of quadratic randomness in CCP. More specifically, let us introduce the following quadratic randomness:
\begin{eqnarray}\label{quadratic-form}
&&c(\bm{\xi},\bm{x})= \frac{1}{2}\bm\xi^{\top}A(\bm{x})\bm\xi+\bm{a}(\bm{x})^{\top}\bm{\xi}+a(\bm{x}),
\end{eqnarray}
where $\bm{\xi}\in\mathbb{R}^m$ is a random vector with probability distribution $p^*$, $\bm{x}\in\mathbb{R}^n$ is the decision vector, $A:\mathbb{R}^{n}\rightarrow\mathbb{R}^{m\times m}$ is a real symmetric matrix, $\bm{a}:\mathbb{R}^{n}\rightarrow\mathbb{R}^{m}$ and $a:\mathbb{R}^{n}\rightarrow\mathbb{R}$. We consider the following CCP:
\begin{eqnarray*}
{\rm QCCP}&\min\limits_{\bm{x}\in \mathscr{X}}&\rho(\bm{x})\\
&{\rm s.t.}& {\rm Pr}_{\sim p^*}\left\{c(\bm{\xi},\bm{x})\leq0\right\}\geq1-\alpha,
\end{eqnarray*}
where $\mathscr{X}\subseteq\mathbb{R}^n$ is a nonempty convex set, $\alpha\in(0,1)$, $\rho:\mathbb{R}^{n}\rightarrow\mathbb{R}$, $c:\mathbb{R}^{m}\times\mathbb{R}^{n}\rightarrow\mathbb{R}$ and ${\rm Pr}_{\sim p^*}$ denotes the probability that is taken with respect to $p^*$.

In the literature, there exist some related works on some special cases of QCCP. \cite{zymler2013} consider a worst-case chance constraint where $p^*$ is in a set of all probability distributions with the same first- and second-moment. They prove that if $c(\bm{\xi},\bm{x})$ is a quadratic function of $\bm{\xi}$, the worst-case QCCP is equivalent to the worst-case CVaR problem. \cite{cui2013} study the optioned portfolio selection problem under the mean-VaR framework, which is a special case of QCCP. In their model, $\bm{\xi}$ is supposed to be normally distributed. Based on the widely used Delta-Gamma-Normal method which can be referred to \cite{hull2009}, they further assume that $c(\bm{\xi},\bm{x})$ follows a normal distribution and then the problem can be reformulated as a second-order cone program (SOCP). \cite{zhu2020} prove that the asymptotic distribution of $c(\bm{\xi},\bm{x})$ remains normal distribution if $\bm{\xi}$ is normally distributed under some mild conditions, which provides the theoretical basis for the Delta-Gamma-Normal method. \cite{kishida2023} study a system control problem with a similar structure to QCCP. They assume that the distribution of $\bm{\xi}$ lies in a probability set with known mean and variance, and consider the worst-case chance constraint, which is further approximated by worst-case CVaR for tractability. It is evident that how to specify the distribution of $\bm{\xi}$ is crucial for solving QCCP.

The procedure of specifying the distribution of underlying random factors is usually called input modeling, which not only affects the choice of solution methods, but also directly affects the rationality of practical application. Therefore, it is meaningful and necessary to develop a unified approach that combines the input modeling with CCP, so that it can be applied to as many cases as possible. In addition, in many application fields, there usually exist substantial observations of the random factors, and thus it is important to estimate the true distribution by learning the data. The mixture distribution model, which is first used by \cite{pearson1894} in the biological field, is a flexible approach for learning the data. The most attractive characteristic of the mixture distribution is its fitting capacity. The density function of mixture distribution, which is the weighted sum of a finite number of density functions, can approximate many density functions within sufficient precision. More properties about the mixture model can be referred to \cite{mclachlan}. Theoretically, \cite{wilson2000} proves that any univariate integrable distribution function can be approximated by Gaussian mixture distribution (GMD) within any precision. Distributions with some common and even unusual features, such as skewness, fat tail, and multimodality can be well modeled by GMD (\cite{marron1992}).

 Recently, the combination of CCP and GMD has already been studied in both theory and application. \cite{chen2018} propose a convex approximation approach to solve the robust CCP where the random parameters are modeled by GMD and the weights on components lie in an uncertainty set based on moment estimation. \cite{hu2021} consider some typical forms of joint and single CCPs where the random parameters are modeled by GMD. They derive the first-order method and spatial branch-and-bound approach to search the local and the global solutions, respectively. In the field of distributed generation, \cite{chen2022} assume that the wind output follows GMD and solve a CCP to determine the optimal power scheduling of generators. \cite{ren2022} use GMD to model the multimodal behaviors of obstacles' uncertain states and use CCP to guarantee the safety of the trajectory planning. 

There are some general scenario approximation methods that can be used to solve QCCP. \cite{campi2005} and \cite{campi2008} prove the lower bound on the number of samples that guarantees the optimal solution of the scenario approach satisfies the chance constraint with a high confidence level. \cite{campi2011} consider sample discarding and further study the lower bound on the number of samples. \cite{luedtke2008} replace the violation probability of the chance constraint with the empirical violation probability, which can be reformulated as a mixed integer program. In this paper, different from these approaches, we derive a unified approach for QCCP based on GMD. More specifically, we use multivariate GMD to model the distribution of the random variables and show that the asymptotic distribution of the quadratic form of those random variables follows a univariate GMD under some mild conditions. And this result is further used to reformulate QCCP as a deterministic program with analytical form and then the branch-and-bound algorithm or the first-order method proposed by \cite{hu2021} can be applied to derive the optimal solution. Furthermore, we show the relationship between the convergence rate and other factors, among which the condition numbers of covariance matrices of GMD's components are highlighted. To reduce the asymptotic approximation error, we suggest to use condition number constrained GMD to model the distribution of random variables. In addition, we show that this condition number constrained GMD can still approximate any integrable density function within any precision. The main contributions of this paper are as follows:

$\centerdot$ This paper proposes a unified approach for solving CCP with quadratic randomness based on GMD. Although the true distribution of the quadratic randomness has no analytic form, an approximated global solution can be obtained by combining the asymptotic distribution approximation and the branch-and-bound algorithm.

$\centerdot$ This paper uncovers the factors that affect the convergence rate of the asymptotic distribution of the quadratic randomness. It is proved that a condition number constrained GMD leads to a faster convergence for the quadratic randomness without loss of fitting capability and flexibility.

$\centerdot$ The proposed approach can be applied to a wide range of areas where data are available. Unlike other commonly used distributions, GMD can learn the characteristics of distribution from the data, e.g., skewness, fat tail, and multimodal shape. Therefore, our model is especially beneficial to those areas where the data are hard to depict.

The rest of this paper is organized as follows. In Section 2, we discuss the asymptotic distribution of the quadratic form of random variables following GMD. We also explore the associated factors that affect the convergence rate, based on which we investigate how to reduce the approximation errors while using the asymptotic distribution as a proxy of the real one. In Section 3, we reformulate the QCCP and apply the branch-and-bound algorithm to solve the optimization problem globally under the assumption that $c(\bm{\xi},\bm{x})$ is a linear function of $\bm{x}$. In Section 4, numerical experiments are executed to verify the theoretical findings and test the effectiveness of the proposed method. The concluding remark is provided in Section 5.

\section{Asymptotic Properties of Quadratic Randomness under GMD}
Due to the universal approximation ability of GMD, we suppose in this paper that $\bm{\xi}$ follows a GMD with $K$ Gaussian components. More specifically, the density function of $\bm{\xi}$ is given by
\begin{eqnarray} \label{GMD}
&&p^*(\bm{z}) = \sum_{i=1}^K\pi_ip(\bm{z}|\bm{\mu}_i,\Sigma_i),
\end{eqnarray}
where the mixture weight $\pi_i>0$, $\sum_{i=1}^K\pi_i=1$ and $p(\bm{z}|\bm{\mu}_i,\Sigma_i)$ (abbreviated as $p_i$) denotes the Gaussian density function with mean $\bm{\mu}_i$ and covariance matrix $\Sigma_i$. Under this setting, we first discuss the asymptotic distribution of the quadratic form of  $\bm{\xi}$. Then we analyze the convergence rate of the asymptotic distribution. Finally, we discuss the factors that affect the error between the asymptotic distribution and the true distribution, based on which we further explore how to reduce the approximation errors.
\subsection{Asymptotic Distribution of Quadratic Randomness}
It will be shown that, under some mild conditions, the asymptotic distribution of the quadratic form of random vector following a GMD remains a GMD. We start the analysis from the first- and second-moment of the quadratic form of a random vector following a GMD in the sequel.
\begin{proposition}\label{moment}
Suppose that $\bm{\xi}$ follows a GMD defined by (\ref{GMD}). Then for a fixed decision vector $\bm{x}$, the mean and variance of $c(\bm{\xi},\bm{x})$ defined by (\ref{quadratic-form}) are given as
\begin{eqnarray*}
E_{p^*}(c(\bm{\xi},\bm{x})) &=& \sum_{i=1}^{K}\pi_iE_{p_i}(c(\bm{\xi},\bm{x}))\\
V_{p^*}(c(\bm{\xi},\bm{x})) &=& \sum_{i=1}^{K}\pi_iV_{p_i}(c(\bm{\xi},\bm{x}))+\sum_{1\leq i<j\leq K}\pi_i\pi_j(E_{p_i}(c(\bm{\xi},\bm{x}))-E_{p_j}(c(\bm{\xi},\bm{x})))^2,
\end{eqnarray*}
where $E_{p_i}(\bm{\xi},\bm{x})$ and $V_{p_i}(\bm{\xi},\bm{x})$ are the mean and variance with respect to the probability distribution $p_i$, which are given by
\begin{eqnarray*}
E_{p_i}(c(\bm{\xi},\bm{x}))&=&\frac{1}{2}{\rm tr}\left(A(\bm x)\Sigma_i\right)+\frac{1}{2}\bm{\mu}_i^{\top}A(\bm x)\bm{\mu}_i+\bm{a}(\bm x)^{\top}\bm{\mu}_i+a(\bm x)\\
V_{p_i}(c(\bm{\xi},\bm{x}))&=&\frac{1}{2}{\rm tr}\left((A(\bm x)\Sigma_i)^2\right)+\left(A(\bm x)\bm{\mu}_i+\bm{a}(\bm x)\right)^{\top}\Sigma_i\left(A(\bm x)\bm{\mu}_i+\bm{a}(\bm x)\right)
\end{eqnarray*}
respectively. Here, ${\rm tr}(\cdot)$ denotes the trace of a matrix.
\end{proposition}

\begin{proof}
See Appendix A.
\end{proof}

Now we turn to the asymptotic distribution of $c(\bm{\xi},\bm{x})$ where $\bm \xi$  follows GMD. For simplicity of reformulation, we omit the decision vector and denote $A(\bm x)$, ${\bm a}(\bm x)$ and $a(\bm x)$ as $A$, ${\bm a}$ and $a$, respectively. 

Without loss of generality, we assume that each $\Sigma _i$, $i\in\{1,\cdots,K\}$ is non-singular. Then we decompose $\Sigma_i$ as $\Sigma_i=\Sigma_i^{\frac{1}{2}}\Sigma_i^{\frac{1}{2}}$, and $\Sigma_i^{\frac{1}{2}}A\Sigma_i^{\frac{1}{2}}$ as $\Sigma_i^{\frac{1}{2}}A\Sigma_i^{\frac{1}{2}}=D_i\Lambda_iD_i^{\top}$ where $\Lambda_i$ is a diagonal matrix with diagonal elements being $\lambda_{i1},\cdots,\lambda_{im}$ and $D_i$ is the associated orthogonal matrix. Assume only the first $h_i$ diagonal elements of $\Lambda_i$ are nonzero, i.e.,
$\lambda_{i1},\lambda_{i2}\cdots,\lambda_{ih_i}\neq 0~ (0<h_i\leq m)$. Then we have
\begin{eqnarray}
c(\bm{\xi},\bm{x})
&=&
\frac{1}{2}\hat{\bm{\xi}}_i^{\top}\Sigma_i^{\frac{1}{2}}A\Sigma_i^{\frac{1}{2}}\hat{\bm{\xi}}_i+\bm{a}^{\top}\Sigma_i^{\frac{1}{2}}\hat{\bm{\xi}}_i+a
~~~~~~~~~~~~~~~~~~~~~~\left(\hat{\bm{\xi}}_i=\Sigma_i^{-\frac{1}{2}}\bm{\xi}\right)\nonumber\\
&=&
\frac{1}{2}\tilde{\bm{\xi}}^{\top}_i\Lambda_i\tilde{\bm{\xi}}_i+\bm{b}_i^{\top}\tilde{\bm{\xi}}_i+a
~~~~~~~~~~~~~~~~~~~~~~~~~~~~~~~~~\left(\tilde{\bm{\xi}}_i=D_i^{\top}\hat{\bm{\xi}}_i,\bm{b}_i=D_i^{\top}\Sigma_i^{\frac{1}{2}}\bm{a}\right)\nonumber\\
&=&
\frac{1}{2}\sum_{j=1}^{h_i}\lambda_{ij}\tilde{\xi}_{ij}^2+\sum_{j=1}^mb_{ij}\tilde{\xi}_{ij}+a\nonumber\\
&=&
\frac{1}{2}\sum_{j=1}^{h_i}\lambda_{ij}\left(\tilde{\xi}_{ij}+\frac{b_{ij}}{\lambda_{ij}}\right)^2+\sum_{j=h_i+1}^{m}b_{ij}\tilde{\xi}_{ij}+c_i
~~~~~\left(c_i=a-\frac{1}{2}\sum_{j=1}^{h_i}\frac{b^2_{ij}}{\lambda_{ij}}\right)\nonumber\\
&=&
\frac{1}{2}\sum_{j=1}^{h_i}\lambda_{ij}\zeta_{ij}^2+\sum_{j=h_i+1}^mb_{ij}\tilde{\xi}_{ij}+c_i.
~~~~~~~~~~~~~~~~~~~\left(\zeta_{ij}=\tilde{\xi}_{ij}+\frac{b_{ij}}{\lambda_{ij}}\right)\label{qua_refor}
\end{eqnarray}

If $\bm{\xi}$ follows the normal distribution $p_i$, then both the covariance matrices of  $\hat{\bm{\xi}}_i$ and $\tilde{\bm{\xi}}_i$ are identity matrix. Thus $\zeta_{ij}$ ($j=1,\cdots,h_i$) and $\tilde{\xi}_{ij}$ ($j=h_i+1,\cdots,m$) are independent normal random variables with unit variances and means $\delta_{ij}=d_{ij}+\frac{b_{ij}}{\lambda_{ij}}$ ($j=1,\cdots,h_i)$ and $d_{ij}$ ($j=h_i+1,\cdots,m$), respectively. Here, $d_{ij}$ is the $j$th element of $D_i^{\top}\Sigma_i^{-\frac{1}{2}}\bm{\mu}_i$. Notice that if $\Sigma_i$ is singular, we can also decompose $c(\bm{\xi},\bm{x})$ into the form of (\ref{qua_refor}), but a more complicated process of reformulation is required.

According to the L\'{e}vy's continuity lemma (\cite{van1998}) which is summarized in Appendix B, the convergence of characteristic function is equivalent to the convergence in distribution. Noting that the above reformulation is the same for each probability distribution $p_i$, by investigating the characteristic function of $\frac{c(\bm{\xi},\bm{x})}{\sqrt{V_{p^*}(c(\bm{\xi},\bm{x}))}}$, we have the following results on the asymptotic distribution of $c(\bm{\xi},\bm{x})$ while $\bm{\xi}$ follows a GMD.

\begin{theorem}\label{asy_GMD}
Suppose that $\bm{\xi}$ follows a GMD defined by (\ref{GMD}), and $h_i\rightarrow+\infty$ as $m\rightarrow+\infty$, $i=1,\cdots,K$. Then
\begin{eqnarray*}
&&\lim\limits_{m\rightarrow+\infty}\left(E_{p*}\left(e^{{\rm i}t\frac{c(\bm{\xi},\bm{x})}{\sigma}}\right)-
\sum_{i=1}^K\pi_ie^{{\rm i}\frac{\mu_i}{\sigma}t-\frac{1}{2}\frac{\sigma_i^2}{\sigma^2}t^2}\right)=0,
\end{eqnarray*}
where $\mu_i=E_{p_i}(c(\bm{\xi},\bm{x}))$, $\sigma_i^2=V_{p_i}(c(\bm{\xi},\bm{x}))$ and $\sigma^2=V_{p^*}(c(\bm{\xi},\bm{x}))$, if
\begin{eqnarray}\label{assumption1}
&&\lim\limits_{m\rightarrow+\infty}\frac{\left(\sum_{j=1}^{h_i}\left(\frac{1}{k}+\delta_{ij}^2\right)\lambda^k_{ij}\right)^{\frac{1}{k}}}
{\left(\sum_{j=1}^{h_i}\left(\frac{1}{2}+\delta_{ij}^2\right)\lambda_{ij}^2\right)^{\frac{1}{2}}}=0
\end{eqnarray}uniformly holds for any $k\geq3$ and for $i=1,\cdots,K$.
\end{theorem}
\begin{proof}
First, we show that
\begin{eqnarray*}
&&\lim\limits_{m\rightarrow+\infty}\left(E_{p_i}\left(e^{{\rm i}t\frac{c(\bm{\xi},\bm{x})}{\sigma}}\right)-e^{{\rm i}\frac{\mu_i}{\sigma}t-\frac{1}{2}\frac{\sigma_i^2}{\sigma^2}t^2}\right)=0,~i=1,\cdots,K.
\end{eqnarray*}
Since the proof of the above equation is the same for each $p_i$, we consider the case $i=1$ without loss of generality and simplify the notations as $\delta_{1j}=\delta_j$, $\zeta_{1j}=\zeta_j$, $\lambda_{1j}=\lambda_j$, $b_{1j}=b_j$, $d_{1j}=d_j$, $c_1=c$ and $h_1=h$. 

Denote $S_h=\prod\limits_{j=1}^hE_{p_1}\left(e^{{\rm i}t\frac{\lambda_{j}\zeta_j^2}{2\sigma}}\right)$. According to (\ref{qua_refor}) and the characteristic function of the normally distributed random variable, if $h<m$, we have 
\begin{eqnarray}\label{P4-1}
E_{p_1}\left(e^{{\rm i}t\frac{c(\bm{\xi},\bm{x})}{\sigma}}\right)&=&
\prod\limits_{j=1}^hE_{p_1}\left(e^{{\rm i}t\frac{\lambda_{j}\zeta_j^2}{2\sigma}}\right)
\prod\limits_{j=h+1}^mE_{p_1}\left(e^{{\rm i}t\frac{b_{j}\tilde{\xi}_j}{\sigma}}\right)
E_{p_1}\left(e^{{\rm i}t\frac{c}{\sigma}}\right)
=S_he^{{\rm i}q_1t-\frac{1}{2}q_2t^2}, 
\end{eqnarray}
where $q_1=\frac{c}{\sigma}+\sum\limits_{j=h+1}^{m}\frac{b_jd_j}{\sigma}$ and $q_2=\sum\limits_{j=h+1}^{m}\frac{b_j^2}{\sigma^2}$. Otherwise, for $h=m$, there is no linear term in $c(\bm{\xi},\bm{x})$, i.e., $E_{p_1}\left(e^{{\rm i}t\frac{c(\bm{\xi},\bm{x})}{\sigma}}\right)=S_he^{{\rm i}\frac{c}{\sigma}t}$, and the proof is a special case of the following one.

Let us first focus on $S_h$ only in the sequel. Since $\bm \xi$ is normally distributed with density $p_1$, $\zeta_j^2$ is a non-central Chi-square random variable and its characteristic function is $(1-{\rm i}2t)^{-\frac{1}{2}}e^{\frac{{\rm i}t}{1-{\rm i}2t}\delta_j^2}$. Thus we have
\begin{eqnarray*}
&&S_h=\prod\limits_{j=1}^hE_{p_1}\left(e^{{\rm i}\frac{\lambda_{j}t}{2\sigma}\zeta_j^2}\right)
=\prod_{j=1}^{h}\left(1-{\rm i}\frac{\lambda_{j}}{\sigma}t\right)^{-\frac{1}{2}}e^{\frac{{\rm i}\frac{\lambda_{j}}{2\sigma}t}{1-{\rm i}\frac{\lambda_{j}}{\sigma}t}\delta_j^2}.
\end{eqnarray*}
Denote $\gamma_j=\frac{\lambda_{j}}{\sigma}$. Taking logarithm on both sides of the above equation derives
\begin{eqnarray}\label{P4-2}
&&{\rm ln}\left(S_h\right)=-\frac{1}{2}\sum_{j=1}^{h}{\rm ln}(1-{\rm i}\gamma_jt)+\frac{1}{2}\sum_{j=1}^{h}\frac{{\rm i}\gamma_jt}{1-{\rm i}\gamma_jt}\delta_j^2.
\end{eqnarray}

Before taking Taylor's expansion of the equation (\ref{P4-2}) with respect to ${\rm i}\gamma_jt$ for further analysis, we need to prove ${\rm i}\gamma_jt$ is small enough. To this end,
in the following we first show $\lim\limits_{m\rightarrow+\infty}\gamma_{j}=0$. For simplicity, we further denote 
\begin{eqnarray*}
&&\tilde{\sigma}_k=\left(\sum_{j=1}^{h}\left(\frac{1}{k}+\delta_{j}^2\right)\lambda_{j}^k\right)^{\frac{1}{k}}, ~~k=1,\cdots,n,\cdots.
\end{eqnarray*}
Notice that $\delta_j=d_j+\frac{b_j}{\lambda_j}$ for $j=1,\cdots,h$ and $\lambda_j=0$ for $j=h+1,\cdots,m$. Then according to some algebraic operations, we have
\begin{eqnarray*}
\tilde{\sigma}_2^2&\leq&\sum_{j=1}^h\left(\frac{1}{2}+\delta_j^2\right)\lambda_{j}^2+\sum_{j=h+1}^mb_j^2=\sum_{j=1}^{m}(\lambda_{j}d_j+b_j)^2+\frac{1}{2}\sum_{j=1}^m\lambda_{j}^2\\
&=&\bm \varsigma_1^{\top}D_1^{\top}D_1\bm \varsigma_1+\frac{1}{2}{\rm tr}\left(\left(\Sigma_1^{\frac{1}{2}}A\Sigma_1^{\frac{1}{2}}\right)^2\right)\\
&=&\left(\Sigma_1^{\frac{1}{2}}A\bm{\mu}_1+\Sigma_1^{\frac{1}{2}}\bm{a}\right)^{\top}\left(\Sigma_1^{\frac{1}{2}}A\bm{\mu}_1+\Sigma_1^{\frac{1}{2}}\bm{a}\right)+\frac{1}{2}{\rm tr}\left(\left(\Sigma_1^{\frac{1}{2}}A\Sigma_1^{\frac{1}{2}}\right)^2\right)\\
&=&(A\bm{\mu}_1+\bm{a})^{\top}\Sigma_1(A\bm{\mu}_1+\bm{a})+\frac{1}{2}{\rm tr}((A\Sigma_1)^2)=\sigma_1^2,
\end{eqnarray*}
where $\bm \varsigma_1=\Lambda_1D_1^{\top}\Sigma_1^{-\frac{1}{2}}\bm{\mu}_1+D_1^{\top}\Sigma_1^{\frac{1}{2}}\bm{a}$. The second equality holds since $d_j$ and $b_j$ are the $j$th element of $D_1^{\top}\Sigma_1^{-\frac{1}{2}}\bm{\mu}_1$ and $D_1^{\top}\Sigma_1^{\frac{1}{2}}\bm{a}$. Now we can reformulate $\gamma_j$ as a product of three terms
\begin{eqnarray*}
&&\gamma_j=\frac{\lambda_j}{\sigma}=\frac{\lambda_j}{\tilde{\sigma}_2}\frac{\tilde{\sigma}_2}{\sigma_1}\frac{\sigma_1}{\sigma},
\end{eqnarray*}
where the second term is obviously bounded. In addition, according to Proposition \ref{moment}, $\sigma^2\geq\sum_{i=1}^{K}\pi_i\sigma^2_i\geq\pi_1\sigma^2_1$, which further implies that $\sigma_1/\sigma\leq1/\sqrt{\pi_1}$ is also bounded. For the first term, since
$\frac{\tilde{\sigma}_4}
{\tilde{\sigma}_2}\geq\left(\frac{1}{4}\right)^{\frac{1}{4}}\frac{|\lambda_{j}|}{\tilde{\sigma}_2}$, according to assumption (\ref{assumption1}), we have
\begin{eqnarray*}
&&0=\lim\limits_{m\rightarrow+\infty}\frac{\tilde{\sigma}_4}
{\tilde{\sigma}_2}\geq\lim\limits_{m\rightarrow+\infty}\left(\frac{1}{4}\right)^{\frac{1}{4}}\frac{|\lambda_{j}|}{\tilde{\sigma}_2}\geq 0,
\end{eqnarray*}
which implies that $\lim\limits_{m\rightarrow+\infty}\frac{\lambda_j}{\tilde{\sigma}_2}=0$. Thus $\lim\limits_{m\rightarrow+\infty}\gamma_j=0$.

Let us return to equation (\ref{P4-2}). For any given $t$, there exists sufficiently large $m$ such that $|{\rm i}\gamma_{j}t|<1$ since $\lim\limits_{m\rightarrow+\infty}\gamma_{j}=0$. According to Taylor's expansion, for $|x|<1$, ${\rm ln}(1-x)=\sum_{n=1}^\infty\frac{-1}{n}x^n$ and $\frac{x}{1-x}=\sum_{n=1}^\infty x^n$. Then for sufficiently large $m$, equation (\ref{P4-2}) can be reformulated as
\begin{eqnarray}\label{prop-e8}
&&{\rm ln}\left(S_h\right)=\frac{1}{2}\sum_{j=1}^{h}\sum_{k=1}^\infty\left(\delta_j^2+\frac{1}{k}\right)({\rm i}\gamma_jt)^k=\frac{1}{2}\sum_{k=1}^\infty\left(\sum_{j=1}^{h}\left(\delta_j^2+\frac{1}{k}\right)\gamma_j^k\right)({\rm i}t)^k.
\end{eqnarray}
Recalling that $\tilde{\sigma}_2\leq\sigma_1$, $\sigma_1\leq\frac{\sigma}{\sqrt{\pi_1}}$ and $\gamma_j=\frac{\lambda_j}{\sigma}$, we have
\begin{eqnarray*}
&&\left|\left(\sum_{j=1}^{h}\left(\delta_j^2+\frac{1}{k}\right)\gamma_j^k\right)^{\frac{1}{k}}\right| =\left|\frac{\tilde{\sigma}_k}{\sigma}\right|\leq
\left|\frac{\tilde{\sigma}_k}{\sqrt{\pi_1}\sigma_1}\right|\leq
\left|\frac{\tilde{\sigma}_k}
{\sqrt{\pi_1}\tilde{\sigma}_2}\right|.
\end{eqnarray*}
According to the assumption (\ref{assumption1}), for any given $t$ and sufficiently small $\epsilon>0$, there exists sufficiently large $m$ irrelavent to $k$ such that $\left|{\rm i}t\left(\sum_{j=1}^{h}\left(\delta_j^2+\frac{1}{k}\right)\gamma_{j}^k\right)^{\frac{1}{k}}\right|<\epsilon$ for each $k\geq3$. Therefore,
\begin{eqnarray*}
&&\left|\sum_{k=3}^{\infty}\left(\sum_{j=1}^{h}\left(\delta_j^2+\frac{1}{k}\right)\gamma_j^k\right)({\rm i}t)^k\right|\leq
\sum_{k=3}^{\infty}\left|\left(\sum_{j=1}^{h}\left(\delta_j^2+\frac{1}{k}\right)\gamma_j^k\right)({\rm i}t)^k\right|
<\sum_{k=3}^{\infty}\epsilon^k=\frac{\epsilon^3}{1-\epsilon},
\end{eqnarray*}
which implies that
\begin{eqnarray}\label{P4-3}
&&\lim\limits_{m\rightarrow+\infty}\sum_{k=3}^{\infty}\left(\sum_{j=1}^{h}\left(\delta_j^2+\frac{1}{k}\right)\gamma_j^k\right)({\rm i}t)^k=0.
\end{eqnarray}

According to equations (\ref{P4-1}) and (\ref{prop-e8}) and the fact that $\sum_{j=1}^{h}\left(\delta_j^2+\frac{1}{k}\right)\gamma_j^k=\frac{\tilde{\sigma}_k^k}{\sigma^k}$, we have
\begin{eqnarray}\label{prop-e7}
{\rm ln}\left(E_{p_1}\left(e^{{\rm i}t\frac{c(\bm{\xi},\bm{x})}{\sigma}}\right)\right)&=&{\rm ln}(S_h)+
{\rm i}q_1t-\frac{1}{2}q_2t^2\nonumber\\
&=&\frac{1}{2}\sum_{k=1}^{\infty}\frac{\tilde{\sigma}_k^k}{\sigma^k}({\rm i}t)^k+{\rm i}q_1t-\frac{1}{2}q_2t^2\nonumber\\
&=&{\rm i}\left(\frac{1}{2}\frac{\tilde{\sigma}_1}{\sigma}+q_1\right)t-\frac{1}{2}\left(\frac{\tilde{\sigma}^2_2}{\sigma^2}+q_2\right)t^2+\frac{1}{2}
\sum_{k=3}^{\infty}\frac{\tilde{\sigma}^k_k}{\sigma^k}({\rm i}t)^k\nonumber\\
&=&{\rm i}\frac{\mu_1}{\sigma}t-\frac{1}{2}\frac{\sigma_1^2}{\sigma^2}t^2+\frac{1}{2}
\sum_{k=3}^{\infty}\frac{\tilde{\sigma}^k_k}{\sigma^k}({\rm i}t)^k,
\end{eqnarray}
where the last equality holds due to
\begin{eqnarray*}
\frac{1}{2}\tilde{\sigma}_1+\sigma q_1
&=&
\frac{1}{2}\sum_{j=1}^{h}(\delta_j^2+1)\lambda_{j}+\sum_{j=h+1}^mb_jd_j+c\\
&=&
\frac{1}{2}\sum_{j=1}^{h}\left(d_j+\frac{b_j}{\lambda_j}\right)^2\lambda_j+\frac{1}{2}\sum_{j=1}^{h}\lambda_{j}+\sum_{j=h+1}^md_jb_j+a-\frac{1}{2}\sum_{j=1}^{h}\frac{b^2_{j}}{\lambda_{j}}\\
&=&\frac{1}{2}\sum_{j=1}^{m}d_j^2\lambda_j+\frac{1}{2}\sum_{j=1}^{m}\lambda_{j}+\sum_{j=1}^md_jb_j+a\\
&=&\frac{1}{2}\bm{\mu}_1^{\top}\Sigma_1^{-\frac{1}{2}}D_1\Lambda D_1^{\top}\Sigma_1^{-\frac{1}{2}}\bm{\mu}_1+\frac{1}{2}{\rm tr}(A\Sigma_1)+\bm{\mu}_1^{\top}\Sigma_1^{-\frac{1}{2}}D_1D_1^{\top}\Sigma_1^{\frac{1}{2}}\bm{a}+a\\
&=&\frac{1}{2}\bm{\mu}_1^{\top}A\bm{\mu}_1+\frac{1}{2}{\rm tr}(A\Sigma_1)+\bm{a}^{\top}\bm{\mu}_1+a\\
&=&\mu_1
\end{eqnarray*}
and  the fact $\tilde{\sigma}_2^2+\sigma^2q_2=\sum_{j=1}^{h}\left(\delta_j^2+\frac{1}{2}\right)\lambda_{j}^2+\sum_{j=h+1}^mb_j^2=\sigma_1^2$.

Combing (\ref{P4-3}) and (\ref{prop-e7}) yields
\begin{eqnarray*}
&&\lim\limits_{m\rightarrow\infty}\left({\rm ln}\left(E_{p_1}\left(e^{{\rm i}t\frac{c(\bm{\xi},\bm{x})}{\sigma}}\right)\right)-\left({\rm i}\frac{\mu_1}{\sigma}t-\frac{1}{2}\frac{\sigma_1^2}{\sigma^2}t^2\right)\right)=0,
\end{eqnarray*}
or equivalently,
\begin{eqnarray*}
&&\lim\limits_{m\rightarrow\infty}\left(E_{p_1}\left(e^{{\rm i}t\frac{c(\bm{\xi},\bm{x})}{\sigma}}\right)-
e^{{\rm i}\frac{\mu_1}{\sigma}t-\frac{1}{2}\frac{\sigma_1^2}{\sigma^2}t^2}\right)=0.
\end{eqnarray*}
Obviously, the above conclusion also holds for $i=2,\cdots,K$. 

Finally, notice that 
$E_{p^*}\left(e^{{\rm i}t\frac{c(\bm{\xi},\bm{x})}{\sigma}}\right)=\sum\limits_{i=1}^K \pi_i E_{p_i}\left(e^{{\rm i}t\frac{c(\bm{\xi},\bm{x})}{\sigma}}\right)$, we have
\begin{eqnarray}\label{proposition2}
&&\lim\limits_{m\rightarrow\infty}\left(E_{p^*}\left(e^{{\rm i}t\frac{c(\bm{\xi},\bm{x})}{\sigma}}\right)-
\sum\limits_{i=1}^K\pi_ie^{{\rm i}\frac{\mu_i}{\sigma}t-\frac{1}{2}\frac{\sigma_i^2}{\sigma^2}t^2}\right)=0.
\end{eqnarray}
The proof is completed.
\end{proof}

As shown in the following corollary, the conditions in Theorem \ref{asy_GMD} are very mild. 
\begin{corollary}\label{cor1}
If for each $i\in\{1,\cdots,K\}$, $h_i\rightarrow+\infty$ as $m\rightarrow+\infty$, and $0<\underline{\lambda}\leq|\lambda_{ij}|\leq\overline{\lambda}<+\infty$, $j=1,\cdots,h_i$, then
equation (\ref{proposition2}) holds.
\end{corollary}
\begin{proof}
The conditions in Theorem \ref{asy_GMD} are the same for each $i=1,\cdots, K$. Again, to simplify the analysis, we consider the case for $i=1$ and omit the related subscripts. Notice that $d_j$ and $b_j$ are only determined by parameters of GMD so that they are naturally bounded. Therefore, there exist  finite positive constants $l$ and $u$ such that $l\leq\delta_j^2=\left(d_{j}+\frac{b_{j}}{\lambda_{j}}\right)^2\leq u$ for $j=1,\cdots,h$. Then, for $k\geq3$, we have 
\begin{eqnarray*}
&&\left|\frac{\left(\sum_{j=1}^{h}\left(\frac{1}{k}+\delta_j^2\right)\lambda_{j}^k\right)^{\frac{1}{k}}}
{\left(\sum_{j=1}^{h}\left(\frac{1}{2}+\delta_j^2\right)\lambda^2_{j}\right)^{\frac{1}{2}}}\right|
\leq
\frac{\left((\frac{1}{k}+u)\sum_{j=1}^{h}|\lambda_{j}|^k\right)^{\frac{1}{k}}}
{\left((\frac{1}{2}+l)\sum_{j=1}^{h}\lambda^2_{j}\right)^{\frac{1}{2}}}
\leq
\frac{(\frac{1}{k}+u)^{\frac{1}{k}}\overline{\lambda}}
{(\frac{1}{2}+l)^{\frac{1}{2}}\underline{\lambda}}h^{\frac{1}{k}-\frac{1}{2}}\leq Mh^{-\frac{1}{6}},
\end{eqnarray*}
where $M>\frac{\left(\frac{1}{k}+u\right)^{\frac{1}{k}}\overline{\lambda}}
{\left(\frac{1}{2}+l\right)^{\frac{1}{2}}\underline{\lambda}}$ is a constant irrelevant to $k$. Notice that $\frac{\left(\frac{1}{k}+u\right)^{\frac{1}{k}}\overline{\lambda}}
{\left(\frac{1}{2}+l\right)^{\frac{1}{2}}\underline{\lambda}}$ is bounded for $k\geq3$ so $M$ can be any number larger than this bound. Then for any $\epsilon>0$, there exists $\overline{h}=\left(\frac{M}{\epsilon}\right)^6$ irrelevant to $k$ such that while $ h>\overline{h}$,
\begin{eqnarray*}
&&\left|\frac{\left(\sum_{j=1}^{h}\left(\frac{1}{k}+\delta_j^2\right)\lambda_{j}^k\right)^{\frac{1}{k}}}
{\left(\sum_{j=1}^{h}\left(\frac{1}{2}+\delta_j^2\right)\lambda^2_{j}\right)^{\frac{1}{2}}}\right|<\epsilon,
\end{eqnarray*}
which implies that $\lim\limits_{m\rightarrow+\infty}\frac{\left(\sum_{j=1}^{h}\left(\frac{1}{k}+\delta_j^2\right)\lambda_{j}^k\right)^{\frac{1}{k}}}
{\left(\sum_{j=1}^{h}\left(\frac{1}{2}+\delta_j^2\right)\lambda^2_{j}\right)^{\frac{1}{2}}}=0$ uniformly holds for $k\geq3$ if $h\rightarrow+\infty$ as $m\rightarrow+\infty$. This verifies the conditions in Theorem \ref{asy_GMD}. The proof is completed.
\end{proof}

In plain language, Theorem \ref{asy_GMD} and Corollary \ref{cor1} conclude that if the number of non-zero eigenvalues of $\Sigma_i^{\frac{1}{2}}A\Sigma_i^{\frac{1}{2}}$ for each $i\in\{1,\cdots,K\}$ are sufficiently large, then the distribution of $c(\bm{\xi},\bm{x})$ can be approximated with the following univariate GMD 
\begin{eqnarray}\label{plain}
&&{p}^*(z) = \sum_{i=1}^K\pi_ip(z|\mu_i,\sigma^2_i).
\end{eqnarray}
Recall that $\mu_i$ and $\sigma^2_i$ are the mean and variance of $c(\bm{\xi},\bm{x})$ with respect to probability distribution $p_i$. 

Specially, when $K=1$, the distribution of the random variables turns to a multivariate Gaussian distribution and the asymptotic distribution of $c(\bm{\xi},\bm{x})$ is a Gaussian distribution, which is consistent with the results of \cite{zhu2020}. If $A=O$, then $c(\bm{\xi},\bm{x})$ becomes a linear function of  $\bm{\xi}$ and this case has already been investigated by \cite{hu2021}.

\subsection{Error Analysis of Asymptotic Approximation}

Now we explore the factors that cause errors when using asymptotic distribution (\ref{plain}) as an approximation of the true distribution of $c(\bm{\xi},\bm{x})$ with finite $m$. Let us start from the conditions proposed in Theorem \ref{asy_GMD}. Notice that the conditions are the same in form for each $i\in\{1,\cdots,K\}$ with respect to $p_i$. Thus we only analyze the case when $K=1$ for simplicity, and the results can be parallelly generalized to the case of $K>1$. 

 For $K=1$, by dropping subscript $i$ again, equation (\ref{qua_refor}) can be simplified as
\begin{eqnarray}\label{qua_refor_normal}
&&c(\bm{\xi},\bm{x})=\frac{1}{2}\sum_{j=1}^{h}\lambda_{j}\zeta_{j}^2+\sum_{j=h+1}^mb_{j}\tilde{\xi}_{j}+c,
\end{eqnarray}
which is constituted by three terms: weighted sum of the square of independent Gaussian random variables, weighted sum of independent Gaussian random variables, and a constant. Furthermore, the first term and the second term are independent of each other and the asymptotic distribution of $c(\bm{\xi},\bm{x})$ is the Gaussian distribution. It is clear that the error of asymptotic approximation is due to the first term. Therefore,  we investigate the asymptotic behaviour of the following general weighted sum of non-central Chi-square variables
\begin{eqnarray}\label{sum_chi}
&&Z_h=\sum_{j=1}^h\omega_jz_j^2,
\end{eqnarray}
where for each $j$, $z_j$ is an independent Gaussian random variable with mean $\alpha_j$ and variance 1,  i.e., $z_j\sim\mathcal{N}(\alpha_j,1)$. We provide its convergence rate of the characteristic function of $Z_h$ in the following theorem.
\begin{theorem}\label{con_rate}
Suppose $\{|\omega_j|\}_{j=1}^{h}$ is a non-zero decreasing sequence. If $\lim\limits_{h\rightarrow+\infty}\frac{\omega_1}{\omega_h}\bar{\alpha}_h^{-\frac{1}{6}}h^{-\frac{1}{6}}=0$, where $\bar{\alpha}_h=\frac{1}{2}+\frac{1}{h}\sum_{j=1}^{h}\alpha_j^2$, then 
\begin{eqnarray*}
&&\left|{\rm ln}\left(E\left(e^{{\rm i}t\overline{Z}_h}\right)\right)-{\rm ln}\left(e^{-\frac{1}{2}t^2}\right)\right|=O\left(\frac{\omega^3_1}{\omega^3_h}\bar{\alpha}_{h}^{-\frac{1}{2}}h^{-\frac{1}{2}}\right)
\end{eqnarray*}
for sufficient large $h$, where $\overline{Z}_h=\frac{Z_h-E(Z_h)}{\sqrt{V(Z_h)}}$.
\end{theorem}

\begin{proof}
According to Proposition 2 of \cite{zhu2020}, the logarithm of characteristic function of $\overline{Z}_h$ is
\begin{eqnarray}\label{C2-1}
&&{\rm ln}\left(E\left(e^{{\rm i}t\overline{Z}_h}\right)\right)=-\frac{1}{2}t^2+\frac{1}{2}\sum_{k=3}^{\infty}{\rm i}^kf_k^kt^k,
\end{eqnarray}
where $f_k=\frac{\left(\sum_{j=1}^h(\alpha_j^2+\frac{1}{k})\omega_j^k\right)^{\frac{1}{k}}}{\left(\sum_{j=1}^h(\alpha_j^2+\frac{1}{2})\omega_j^2\right)^{\frac{1}{2}}}$. We have
\begin{eqnarray}\label{f_k_ineq}
|f_k|=\left|\frac{\left(\sum_{j=1}^h(\alpha_j^2+\frac{1}{k})\omega_j^k\right)^{\frac{1}{k}}}{\left(\sum_{j=1}^h(\alpha_j^2+\frac{1}{2})\omega_j^2\right)^{\frac{1}{2}}}\right|&\leq&\frac{\left(\sum_{j=1}^h(\alpha_j^2+\frac{1}{k})|\omega_j|^k\right)^{\frac{1}{k}}}{\left(\sum_{j=1}^h(\alpha_j^2+\frac{1}{2})|\omega_j|^2\right)^{\frac{1}{2}}}\leq\frac{\left(\sum_{j=1}^h(\alpha_j^2+\frac{1}{k})|\omega_1|^k\right)^{\frac{1}{k}}}{\left(\sum_{j=1}^h(\alpha_j^2+\frac{1}{2})|\omega_h|^2\right)^{\frac{1}{2}}}\nonumber\\
&\leq& \left|\frac{\omega_1}{\omega_h}\right|\frac{(\bar{\alpha}_{h}+\frac{1}{k})^{\frac{1}{k}}}{(\bar{\alpha}_{h}+\frac{1}{2})^{\frac{1}{2}}}h^{\frac{1}{k}-\frac{1}{2}}\leq\left|\frac{\omega_1}{\omega_h}\right|\bar{\alpha}_h^{\frac{1}{k}-\frac{1}{2}}h^{\frac{1}{k}-\frac{1}{2}}.
\end{eqnarray}
Denote $g_k=\frac{f_k}{\left|\frac{\omega_1}{\omega_h}\right|\bar{\alpha}_h^{-\frac{1}{6}}h^{-\frac{1}{6}}}$. We further have
\begin{eqnarray}
&&|g_k|\leq \bar{\alpha}_h^{\frac{1}{k}-\frac{1}{3}}h^{\frac{1}{k}-\frac{1}{3}}.\label{C2-2}
\end{eqnarray}
Notice that $\left|\frac{\omega_1}{\omega_h}\right|\geq1$ and $\bar{\alpha}_hh\geq \frac{1}{2}h\geq1$ if $h>2$. Then we have $\bar{\alpha}_h^{\frac{1}{k}-\frac{1}{3}}h^{\frac{1}{k}-\frac{1}{3}}\leq \bar{\alpha}_h^{-\frac{1}{12}}h^{-\frac{1}{12}}$ for any $k\geq4$ and $h>2$, which implies that $g_k$ uniformly converges to 0 for $k\geq4$. 

Furthermore, the assumption $\lim\limits_{h\rightarrow+\infty}\frac{\omega_1}{\omega_h}\bar{\alpha}_h^{-\frac{1}{6}}h^{-\frac{1}{6}}=0$ implies that $\left|\frac{\omega_1}{\omega_h}\right|\bar{\alpha}_h^{-\frac{1}{6}}h^{-\frac{1}{6}}<1$ for sufficient large $h$. Therefore, given $t$, for any $k\geq4$ and any small enough $\epsilon>0$, there exists a sufficiently large $h$ which is irrelevant to $k$ such that
\begin{eqnarray*}
&&\frac{\left|{\rm i}f_kt\right|}{\left|\frac{\omega_1}{\omega_h}\right|\bar{\alpha}_h^{-\frac{1}{6}}h^{-\frac{1}{6}}}<\epsilon \mbox{ and } \frac{\left|{\rm i}^kf_k^kt^k\right|}{\left|\frac{\omega_1}{\omega_h}\right|^3\bar{\alpha}_h^{-\frac{1}{2}}h^{-\frac{1}{2}}}
<\frac{\left|{\rm i}^kf_k^kt^k\right|}{\left|\frac{\omega_1}{\omega_h}\right|^k\bar{\alpha}_h^{-\frac{k}{6}}h^{-\frac{k}{6}}}
<\epsilon^k.
\end{eqnarray*}
Consequently, 
\begin{eqnarray*}
&&\frac{\left|\sum\limits_{k=4}^{\infty}{\rm i}^kf^k_kt^k\right|}{\left|\frac{\omega_1}{\omega_h}\right|^3\bar{\alpha}_h^{-\frac{1}{2}}h^{-\frac{1}{2}}}
<\frac{\sum\limits_{k=4}^{\infty}\left|{\rm i}^kf^k_kt^k\right|}{\left|\frac{\omega_1}{\omega_h}\right|^3\bar{\alpha}_h^{-\frac{1}{2}}h^{-\frac{1}{2}}}
<\sum_{k=4}^{\infty}\epsilon^k
=\frac{\epsilon^4}{1-\epsilon},
\end{eqnarray*}
which implies that $\frac{\left|\sum_{k=4}^{\infty}{\rm i}^kf^k_kt^k\right|}{\left|\frac{\omega_1}{\omega_h}\right|^3\bar{\alpha}_h^{-\frac{1}{2}}h^{-\frac{1}{2}}}$ converges to zero as $h$ increases to infinity, i.e.,
\begin{eqnarray}\label{prop3-1}
&&\left|\sum_{k=4}^{\infty}{\rm i}^kf_k^kt^k\right|=o\left(\frac{\omega_1^3}{\omega_h^3}\bar{\alpha}_h^{-\frac{1}{2}}h^{-\frac{1}{2}}\right).
\end{eqnarray}

For $k=3$, according to (\ref{C2-2}), $|g_3|\leq 1$. It is not difficult to see from (\ref{f_k_ineq}) that the equality $|g_3|= 1$ can be achieved under some special case. Thus we have
\begin{eqnarray}\label{prop3-2}
&&|{\rm i}^3f_3^3t^3|=O\left(\frac{\omega_1^3}{\omega_h^3}\bar{\alpha}_h^{-\frac{1}{2}}h^{-\frac{1}{2}}\right).
\end{eqnarray}
 
 Combing (\ref{C2-1}), (\ref{prop3-1}) and (\ref{prop3-2}) yields
\begin{eqnarray*}
\left|{\rm ln}\left(E\left(e^{it\overline{Z}_h}\right)\right)-{\rm ln}\left(e^{-\frac{1}{2}t^2}\right)\right|=\left|\frac{1}{2}
\sum_{k=3}^{\infty}{\rm i}^kf_k^kt^k\right|&\leq&
\left|\frac{1}{2}{\rm i}^3f_3^3t^3\right|+\left|\frac{1}{2}
\sum_{k=4}^{\infty}{\rm i}^kf_k^kt^k\right|\\
&=&
O\left(\frac{\omega_1^3}{\omega_h^3}\bar{\alpha}_h^{-\frac{1}{2}}h^{-\frac{1}{2}}\right)+o\left(\frac{\omega_1^3}{\omega_h^3}\bar{\alpha}_h^{-\frac{1}{2}}h^{-\frac{1}{2}}\right)\\
&=&
O\left(\frac{\omega_1^3}{\omega_h^3}\bar{\alpha}_h^{-\frac{1}{2}}h^{-\frac{1}{2}}\right).
\end{eqnarray*}
Similarly, we have
\begin{eqnarray*}
\left|{\rm ln}\left(E\left(e^{it\overline{Z}_h}\right)\right)-{\rm ln}\left(e^{-\frac{1}{2}t^2}\right)\right|&\geq&
\left|\frac{1}{2}{\rm i}^3f_3^3t^3\right|-\left|\frac{1}{2}
\sum_{k=4}^{\infty}{\rm i}^kf_k^kt^k\right|
=
O\left(\frac{\omega_1^3}{\omega_h^3}\bar{\alpha}_h^{-\frac{1}{2}}h^{-\frac{1}{2}}\right).
\end{eqnarray*}The proof is completed.
\end{proof}

Theorem \ref{con_rate} indicates that, for given $\bm\omega$ and $\bm\alpha$, the distribution of the weighted sum of non-central Chi-square random variables converges to a Gaussian distribution as the number of random variables increases to infinity, and the convergence rate is $O(1/\sqrt{h})$. It also says that the difference between the characteristic functions of $\overline{Z}_h$ and its asymptotic distribution also depends on the weight vector $\bm\omega$ and the mean vector $\bm\alpha$. The smaller the differences among the absolute values of weights $\omega_i$'s and the larger the absolute values of means $\alpha_i$'s, the faster the convergence. Since this conclusion holds for each probability distribution $p_i$, these results remain true for the case of GMD with $K>1$.

\subsection{Error Reduction of Asymptotic Approximation}

There are two types of errors arising from the asymptotic approximation: {\it underlying fitness error} and {\it asymptotic approximation error}. Specifically, the underlying fitness error is the one generated by the approximation of the distribution of $\bm \xi$ with a GMD, while the asymptotic approximation error denotes the one caused by the approximation of the distribution of $c(\bm{\xi},\bm{x})$ with the associated asymptotic distribution based on the underlying GMD. 

Comparing equations (\ref{qua_refor_normal}) and (\ref{sum_chi}) yields
\begin{eqnarray*}
&&\alpha_j \sim d_j+\frac{b_j}{\lambda_j} \mbox{ and } \omega_j\sim\lambda_j,~ j=1,\cdots,h,
\end{eqnarray*}
where $d_j$ and $b_j$ are the $j$th elements of $D^{\top}\Sigma^{-\frac{1}{2}}\bm{\mu}$ and $D^{\top}\Sigma^{\frac{1}{2}}\bm{a}$. By Theorem \ref{con_rate}, the larger the absolute values of $\left(d_j+\frac{b_j}{\lambda_j}\right)$'s and the smaller the ratio of the largest to the smallest absolute value of non-zero eigenvalues of matrix $\Sigma^{\frac{1}{2}}A\Sigma^{\frac{1}{2}}$, the smaller the asymptotic approximation error.
Essentially, the error is determined by distribution parameters $\bm\mu$ and $\Sigma$ of $\bm\xi$,  $A$ and $\bm a$. Given the decision vector $\bm x$, $A$ and $\bm{a}$ are two constants, and the error is only determined by parameters $\bm{\mu}$ and $\Sigma$. Thus we mainly discuss why and how one can reduce the error of asymptotic approximation by properly setting parameters $\bm\mu$ and $\Sigma$ in the sequel.         

Suppose $\bm \mu=\bm 0$ and $\bm a=\bm 0$. Then $d_j+\frac{b_j}{\lambda_j}=0~(i=1,\cdots,h)$, which implies the lowest convergence rate associated with $\bm\alpha$ by Theorem \ref{con_rate}. Thus a larger deviation between the mean vector $\bm{\mu}$ and the origin is more likely to lead to faster convergence. However, given different values of $A$ and $\bm{a}$, $\bm{\mu}$ has different effects on the asymptotic approximation error. Therefore, error reduction may not be achieved by manipulating the mean vector $\bm{\mu}$, and we focus on the discussion of $\Sigma$ in the following. 

Since the number of non-zero eigenvalues of matrix $\Sigma^{\frac{1}{2}}A\Sigma^{\frac{1}{2}}$ instead of its singularity affects our analysis, for simplicity we assume that $\Sigma^{\frac{1}{2}}A\Sigma^{\frac{1}{2}}$ is non-singular, or equivalently, both $A$ and $\Sigma$ are non-singular, in the following analysis. As to the ratio of the maximum to the minimum of the absolute values of non-zero eigenvalues of matrix $\Sigma^{\frac{1}{2}}A\Sigma^{\frac{1}{2}}$,  we have the following proposition where we denote by $|\lambda_j(\cdot)|$ the $j$th largest absolute value of eigenvalues of an $m$-dimensional non-singular matrix for simplicity.

\begin{proposition}\label{con_num}
Suppose both $A$ and $\Sigma$ are non-singular. Then
\begin{eqnarray*}
&&
\left|\frac{\lambda_1(\Sigma^{\frac{1}{2}}A\Sigma^{\frac{1}{2}})}{\lambda_m(\Sigma^{\frac{1}{2}}A\Sigma^{\frac{1}{2}})}\right|\leq
\left|\frac{\lambda_1(A)}{\lambda_m(A)}\right|\frac{\lambda_1(\Sigma)}{\lambda_m(\Sigma)}.
\end{eqnarray*}
\end{proposition}
\begin{proof}
Notice that $\Sigma^{\frac{1}{2}}$ is a definite matrix satisfying $\Sigma=\Sigma^{\frac{1}{2}}\Sigma^{\frac{1}{2}}$. Referring to Theorem 8.1.17 of \cite{golub2013}, we have
\begin{eqnarray*}
&&0<|\lambda_i(A)\lambda_m(\Sigma)|\leq|\lambda_i(\Sigma^{\frac{1}{2}}A\Sigma^{\frac{1}{2}})|\leq|\lambda_i(A)\lambda_1(\Sigma)|
\end{eqnarray*}
for $i=1,\cdots,m$, which indicates that
\begin{eqnarray*}
&&\left|\frac{\lambda_1(\Sigma^{\frac{1}{2}}A\Sigma^{\frac{1}{2}})}{\lambda_m(\Sigma^{\frac{1}{2}}A\Sigma^{\frac{1}{2}})}\right|\leq
\left|\frac{\lambda_1(A)\lambda_1(\Sigma)}{\lambda_m(A)\lambda_m(\Sigma)}\right|=
\left|\frac{\lambda_1(A)}{\lambda_m(A)}\right|\frac{\lambda_1(\Sigma)}{\lambda_m(\Sigma)}.
\end{eqnarray*}
The proof is completed.
\end{proof}

By comparison of equation (\ref{sum_chi}) and the first term of equation (\ref{qua_refor_normal}), Theorem \ref{con_rate} and Proposition \ref{con_num} imply that when the ratio of the largest to the smallest absolute value of eigenvalues of matrix $A$ and/or $\Sigma$ are relatively small, the distribution of $c(\bm \xi,\bm x)$ can converge to the associated asymptotic distribution relatively easily.

An intuition derived from the aforementioned analysis is that we can constrain $\frac{\lambda_1(\Sigma_i)}{\lambda_m(\Sigma_i)}$'s, the condition numbers of covariance matrices of Gaussian distributions that constitute the GMD to reduce asymptotic approximation error. Remind that our original motivation for selecting GMD is because it has strong fitting capacity to the real distribution of $\bm\xi$ which can reduce the underlying fitness error. We show in the following that this purpose can also be achieved even if the condition numbers of covariance matrices are restricted.

\begin{theorem}\label{fitness}
Suppose $f(\bm{z}):\mathbb{R}^m\rightarrow\mathbb{R}$ is a non-negative Riemann integrable function that satisfies $\int_{\mathbb{R}^m}f(\bm{z})d\bm{z}=1$. Then for any given $\delta>0$ and $q\geq1$, there exists a Gaussian mixture density function
\begin{eqnarray}
&&p^*(\bm{z})=\sum_{i=1}^{K}\pi_ip(\bm{z}|\bm{\mu}_i,\Sigma_i),
\end{eqnarray}
where $\pi_i>0$, $i=1,\cdots,K$, and $\frac{\lambda_1(\Sigma_i)}{\lambda_m(\Sigma_i)}\leq q$, $i=1,\cdots,K$, such that
\begin{eqnarray*}
&&\int_{\mathbb{R}^m}|f(\bm{z})-p^*(\bm{z})|d\bm{z}<\delta.
\end{eqnarray*}
\end{theorem}
\begin{proof}
According to the assumption $\int_{\mathbb{R}^m}f(\bm{z})d\bm{z}=1$, $f(\bm{z})$ can be approximated by the finite sum of some simple functions:
\begin{eqnarray*}
&&f_N(\bm{z})=\sum_{i=1}^N\underline{f}_i\mathbf{1}_{\mathscr{I}_i},
\end{eqnarray*}
where $\mathscr{I}_i=\prod_{j=1}^m(a_{ij},b_{ij}]$ $(b_{ij}>a_{ij})$, $i=1,\cdots,N$, are mutually disjoint cubes, $\underline{f}_i=\inf\{f(\bm{z}):\bm{z}\in \mathscr{I}_i\}$ and $\mathbf{1}_{\mathscr{I}_i}$ is the indicator function. Without considering those small cubes containing points whose function values are equal to zeros, we further assume $\underline{f}_i>0$. Here, $N$ is chosen to be sufficiently large such that
\begin{eqnarray}\label{P5-1}
&&\int_{\mathbb{R}^m}|f(\bm{z})-f_N(\bm{z})|d\bm{z}<\frac{\delta}{2}.
\end{eqnarray}

In the following, we show that given the number of cubes $N$ and the partition $\{\mathscr{I}_i\}_{i=1}^{N}$, $f_N(\bm{z})$ can be approximated by a Gaussian mixture density function with an error smaller than $\frac{\delta}{2}$. And this is done by approximating each simple function $\underline{f}_i\mathbf{1}_{\mathscr{I}_i}$ with a convolution, and then approximating the convolution with a Gaussian mixture density function. 

For simplicity, we denote the simple function as 
\begin{eqnarray*}
&&\chi_{\mathscr{I}_i}(\bm{z})=\underline{f}_i\mathbf{1}_{\mathscr{I}_i}(\bm{z})=\left\{\begin{array}{cc}\underline{f}_i&\bm{z}\in \mathscr{I}_i\\0&{\rm otherwise}\\ \end{array}\right.
\end{eqnarray*}
and a Gaussian density function with zero mean vector and covariance matrix $\Sigma_i$ as $p(\bm{z}|\bm{0},\Sigma_i)=\frac{1}{(2\pi)^{\frac{m}{2}}|\Sigma_i|^{\frac{1}{2}}}e^{-\frac{1}{2}\bm{z}^{\top}\Sigma_i^{-1}\bm{z}}$.
Consider the convolution of $\chi_{\mathscr{I}_i}(\cdot)$ with $p(\cdot|\bm{0},\Sigma_i)$ which is defined by
\begin{eqnarray}\label{P5-8}
&&\chi_{\mathscr{I}_i}\ast p(\bm{z}|\bm{0},\Sigma_i)=\int_{\mathscr{I}_i}\chi_{\mathscr{I}_i}(\bm{y})p(\bm{z}-\bm{y}|\bm{0},\Sigma_i)d\bm{y}=
\underline{f}_i\int_{\mathscr{I}_i}p(\bm{z}-\bm{y}|\bm{0},\Sigma_i)d\bm{y}.
\end{eqnarray}
There are two properties of this convolution. First, 
\begin{eqnarray}\label{P5-2}
&&\int_{\mathbb{R}^m}\chi_{\mathscr{I}_i}\ast p(\bm{z}|\bm{0},\Sigma_i)d\bm{z}=
\int_{\mathscr{I}_i}\int_{\mathbb{R}^m}\underline{f}_ip(\bm{z}-\bm{y}|\bm{0},\Sigma_i)d\bm{z}d\bm{y}=\int_{\mathscr{I}_i}\underline{f}_id\bm{y}=|\mathscr{I}_i|\underline{f}_i, 
\end{eqnarray}
where $|\mathscr{I}_i|$ is the volume of $\mathscr{I}_i$. The order of integration in (\ref{P5-2}) is reversed since the integral domains of $\bm{z}$ and $\bm{y}$ are irrelevant. Second, for any $\bm{z}\in \mathscr{I}_i$, 
\begin{eqnarray}\label{P5-3}
&&\chi_{\mathscr{I}_i}(\bm{z})=\underline{f}_i\geq\chi_{\mathscr{I}_i}\ast p(\bm{z}|\bm{0},\Sigma_i).
\end{eqnarray}
Combing with (\ref{P5-2})-(\ref{P5-3}), we can derive the integral absolute error between the simple function and the corresponding convolution
\begin{eqnarray}\label{P5-6}
&&\int_{\mathbb{R}^m}|\chi_{\mathscr{I}_i}(\bm{z})-\chi_{\mathscr{I}_i}\ast p(\bm{z}|\bm{0},\Sigma_i)|d\bm{z}\nonumber\\
&&=\int_{{\mathscr{I}_i}}\left(\chi_{\mathscr{I}_i}(\bm{z})-\chi_{\mathscr{I}_i}\ast p(\bm{z}|\bm{0},\Sigma_i)\right)d\bm{z}+\int_{\mathbb{R}^m\backslash{\mathscr{I}_i}}\chi_{\mathscr{I}_i}\ast p(\bm{z}|\bm{0},\Sigma_i)d\bm{z}\nonumber\\
&&=\int_{{\mathscr{I}_i}}\chi_{\mathscr{I}_i}(\bm{z})d\bm{z}+\int_{\mathbb{R}^m}\chi_{\mathscr{I}_i}\ast  p(\bm{z}|\bm{0},\Sigma_i)d\bm{z}-2\int_{{\mathscr{I}_i}}\chi_{\mathscr{I}_i}\ast  p(\bm{z}|\bm{0},\Sigma_i)d\bm{z}\nonumber\\
&&=2\underline{f}_i|\mathscr{I}_i|-2\int_{{\mathscr{I}_i}}\chi_{\mathscr{I}_i}\ast  p(\bm{z}|\bm{0},\Sigma_i)d\bm{z}=2\underline{f}_i|\mathscr{I}_i|-2\int_{{\rm int}{\mathscr{I}_i}}\chi_{\mathscr{I}_i}\ast  p(\bm{z}|\bm{0},\Sigma_i)d\bm{z},
\end{eqnarray}
where ${\rm int}\mathscr{I}_i=\prod_{j=1}^m(a_{ij},b_{ij})$ is the set of interior point with respect to $\mathscr{I}_i$. This absolute error depends on the integral of the convolution on $\mathscr{I}_i$. In the following we show the condition on which the integral absolute error converges to zero, or equivalently, the integral of the convolution converges to $\underline{f}_i|\mathscr{I}_i|$.

Recalling (\ref{P5-8}), the convolution can be rewritten as
\begin{eqnarray*}
\chi_{\mathscr{I}_i}\ast p(\bm{z}|\bm{0},\Sigma_i)&=&\underline{f}_i\int_{\mathscr{I}_i}\frac{1}{(2\pi)^{\frac{m}{2}}|\Sigma_i|^{\frac{1}{2}}}
e^{-\frac{1}{2}(\bm{y}-\bm{z})^{\top}\Sigma_i^{-1}(\bm{y}-\bm{z})}d\bm{y}\\
&=&
\underline{f}_i\int_{\mathscr{D}_{\bm{z}}}\frac{1}{(2\pi)^{\frac{m}{2}}|\Sigma_i|^{\frac{1}{2}}}
e^{-\frac{1}{2}\bm{x}^{\top}\Sigma_i^{-1}\bm{x}}|J_{\bm{y},\bm{x}}|d\bm{x}\\
&=&
\underline{f}_i\int_{\mathscr{D}_{\bm{z}}}\frac{1}{(2\pi)^{\frac{m}{2}}|\Sigma_i|^{\frac{1}{2}}}
e^{-\frac{1}{2}\bm{x}^{\top}\Sigma_i^{-1}\bm{x}}d\bm{x},
\end{eqnarray*}
where $\bm{x}=\bm{y}-\bm{z}$, $J_{\bm{y},\bm{x}}$ is the corresponding Jacobi matrix, i.e., the derivative matrix of $\bm{y}$ with respect to $\bm{x}$, which is equal to the identity matrix here, and $\mathscr{D}_{\bm{z}}=\{\bm{x}:\bm{x}+\bm{z}\in\mathscr{I}_i\}$. Decompose $\Sigma_i$ as $\Sigma_i=Q\Lambda Q^{\top}$, where $\Lambda$ is a diagonal matrix with entries $\lambda_1,\cdots,\lambda_m$ and $Q$ is an orthogonal matrix, and further denote $\mathscr{F}=\{\bm{u}:\bm{u}=Q^{\top}\bm{x},\bm{x}\in\mathscr{D}_z\}$. For any $\bm{z}\in{\rm int}\mathscr{I}_i$, we have that $\bm{0}\in{\rm int}\mathscr{D}_{\bm{z}}$. Note that $\mathscr{F}$ is the orthogonal transformation of $\mathscr{D}_{\bm{z}}$. Thus $\bm{0}\in{\rm int}\mathscr{F}$. Therefore, there exists a subset $\mathscr{F}_1\subseteq\mathscr{F}$ and $\mathscr{F}_1=\prod_{i=1}^m(-\tau_i,\tau_i]$ with sufficiently small $\tau_i>0$, $i=1,\cdots,m$. Then we have that for any $\bm{z}\in{\rm int}\mathscr{I}_i$,
\begin{eqnarray*}
\chi_{\mathscr{I}_i}\ast p(\bm{z}|\bm{0},\Sigma)&=&\underline{f}_i\int_{\mathscr{D}_{\bm{z}}}\frac{1}{(2\pi)^{\frac{m}{2}}|\Sigma_i|^{\frac{1}{2}}}
e^{-\frac{1}{2}\bm{x}^{\top}\Sigma_i^{-1}\bm{x}}d\bm{x}\\
&=&\underline{f}_i\int_{\mathscr{F}}\frac{1}{(2\pi)^{\frac{m}{2}}|\Lambda|^{\frac{1}{2}}}
e^{-\frac{1}{2}\bm{u}^{\top}\Lambda^{-1}\bm{u}}|J_{\bm{x},\bm{u}}|d\bm{u}\\
&\geq&\underline{f}_i\int_{\mathscr{F}_1}\frac{1}{(2\pi)^{\frac{m}{2}}|\Lambda|^{\frac{1}{2}}}
e^{-\frac{1}{2}\bm{u}^{\top}\Lambda^{-1}\bm{u}}d\bm{u}\\
&=&\underline{f}_i\prod_{i=1}^m\int_{-\tau_i}^{\tau_i}\frac{1}{\sqrt{2\pi\lambda_i}}
e^{-\frac{1}{2\lambda_i}u_i^2}du_i\\
&=&\underline{f}_i\prod_{i=1}^m\int_{-\frac{\tau_i}{\sqrt{\lambda_i}}}^{\frac{\tau_i}{\sqrt{\lambda_i}}}\frac{1}{\sqrt{2\pi}}
e^{-\frac{1}{2}u_i^2}du_i,
\end{eqnarray*}
where $\bm{u}=Q^{\top}\bm{x}$, $J_{\bm{x},\bm{u}}=Q$ and $|J_{\bm{x},\bm{u}}|=|Q|=1$. Given $\tau_i$, $i=1,\cdots,m$,  if all the eigenvalues $\lambda_i$, $i=1,\cdots,m$, are small enough, then $\chi_{\mathscr{I}_i}\ast p(\bm{z}|\bm{0},\Sigma)$ converges to $\underline{f}_i$. Therefore, there exists a sufficiently small $\bar{\lambda}>0$ and a covariance matrix $\Sigma_i$ such that if
\begin{eqnarray*}\label{P5-7}
&&\bar\lambda>\lambda_i>\frac{\bar\lambda}{q},~i=1,\cdots,m, 
\end{eqnarray*}
then
\begin{eqnarray*}
&&\chi_{\mathscr{I}_i}\ast p(\bm{z}|\bm{0},\Sigma_i)\geq\underline{f}_i\left(1-\frac{\delta}{8N|\mathscr{I}_i|\underline{f}_i}\right).
\end{eqnarray*}
Note that the above inequality holds for $\bm{z}\in{\rm int}\mathscr{I}_i$. Integrating both side of the above inequality on ${\rm int}\mathscr{I}_i$ and combining with (\ref{P5-2}) we have
\begin{eqnarray*}
&&|\mathscr{I}_i|\underline{f}_i\geq\int_{{\rm int}\mathscr{I}_i}\chi_{\mathscr{I}_i}\ast p(\bm{z}|\bm{0},\Sigma_i)d\bm{z}\geq|\mathscr{I}_i|\underline{f}_i\left(1-\frac{\delta}{8N|\mathscr{I}_i|\underline{f}_i}\right).
\end{eqnarray*}
Combining the above inequality with (\ref{P5-6}) further derives
\begin{eqnarray}\label{P5-4}
&&\int_{\mathbb{R}^m}|\chi_{\mathscr{I}_i}(\bm{z})-\chi_{\mathscr{I}_i}\ast p(\bm{z}|\bm{0},\Sigma_i)|d\bm{z}=2\underline{f}_i|\mathscr{I}_i|-2\int_{{\rm int}{\mathscr{I}_i}}\chi_{\mathscr{I}_i}\ast  p(\bm{z}|\bm{0},\Sigma_i)d\bm{z}<\frac{\delta}{4N}.
\end{eqnarray}

We have proved that any simple function can be approximated by a corresponding convolution function. In the following we prove that the convolution function can be approximated by a Gaussian mixture density function.

According to (\ref{P5-2}), the convolution $\chi_{\mathscr{I}_i}\ast p(\bm{z}|\bm{0},\Sigma_i)$ is integrable in $\mathbb{R}^m$. Therefore, there exists a compact set $\mathscr{B}$ with sufficiently large volume such that
\begin{eqnarray*}
&&\int_{\mathbb{R}^m\backslash\mathscr{B}}\chi_{\mathscr{I}_i}\ast p(\bm{z}|\bm{0},\Sigma_i)d\bm{z}<\frac{\delta}{16N}.
\end{eqnarray*}
According to Lemma \ref{fitness_lemma} in the Appendix C, given $\mathscr{B}$, there exists a partition $\mathscr{I}_i=\bigcup_{j=1}^{n_i}\mathscr{I}_{ij}$ and a function
\begin{eqnarray*}
&&\chi_{\mathscr{I}_i}\ast p_{n_i}(\bm{z})=\sum_{j=1}^{n_i}|\mathscr{I}_{ij}|\chi_{\mathscr{I}_i}(\bm{y}_{ij})p(\bm{z}-\bm{y}_{ij}|\bm{0},\Sigma_i)=\sum_{j=1}^{n_i}\underline{f}_i|\mathscr{I}_{ij}|p(\bm{z}-\bm{y}_{ij}|\bm{0},\Sigma_i),
\end{eqnarray*}
where $\bm{y}_{ij}\in \mathscr{I}_{ij}$, such that
\begin{eqnarray*}
&&\int_{\mathscr{B}}\left|\chi_{\mathscr{I}_i}\ast p(\bm{z}|\bm{0},\Sigma_i)-\chi_{\mathscr{I}_i}\ast p_{n_i}(\bm{z})\right|d\bm{z}<\frac{\delta}{16N}.
\end{eqnarray*}
In addition, the above inequality also implies that 
\begin{eqnarray*}
\int_{\mathbb{R}^m\backslash\mathscr{B}}\chi_{\mathscr{I}_i}\ast p_{n_i}(\bm{z})d\bm{z}&=&1-\int_{\mathscr{B}}\chi_{\mathscr{I}_i}\ast p_{n_i}(\bm{z})d\bm{z}\\
&<&1-\left(\int_{\mathscr{B}}\chi_{\mathscr{I}_i}\ast p(\bm{z}|\bm{0},\Sigma_i)d\bm{z}-\frac{\delta}{16N}\right)\\
&=&\int_{\mathbb{R}^m\backslash\mathscr{B}}\chi_{\mathscr{I}_i}\ast p(\bm{z}|\bm{0},\Sigma_i)d\bm{z}+\frac{\delta}{16N}\\
&<&\frac{\delta}{8N}.
\end{eqnarray*}
Finally we have
\begin{eqnarray}\label{P5-5}
&&\int_{\mathbb{R}^m}\left|\chi_{\mathscr{I}_i}\ast p(\bm{z}|\bm{0},\Sigma_i)-\chi_{\mathscr{I}_i}\ast p_{n_i}(\bm{z})\right|d\bm{z}\nonumber\\
&&<\int_{\mathscr{B}}\left|\chi_{\mathscr{I}_i}\ast p(\bm{z}|\bm{0},\Sigma_i)-\chi_{\mathscr{I}_i}\ast p_{n_i}(\bm{z})\right|d\bm{z}+\int_{\mathbb{R}^m\backslash\mathscr{B}}\chi_{\mathscr{I}_i}\ast p(\bm{z}|\bm{0},\Sigma_i)d\bm{z}+\int_{\mathbb{R}^m\backslash\mathscr{B}}\chi_{\mathscr{I}_i}\ast p_{n_i}(\bm{z})d\bm{z}\nonumber\\
&&<\frac{\delta}{4N}.
\end{eqnarray}

We have proved that the convolution can be approximated by a Gaussian mixture density function. In the following we combine all of the above to complete the final step of the proof. 

Combining (\ref{P5-5}) with (\ref{P5-4}), we derive
\begin{eqnarray}\label{P5-13}
&&\int_{\mathbb{R}^m}\left|\chi_{\mathscr{I}_i}(\bm{z})-\sum_{j=1}^{n_i}\underline{f}_i|\mathscr{I}_{ij}|p(\bm{z}-\bm{y}_{ij}|\bm{0},\Sigma_i)\right|d\bm{z}<\frac{\delta}{2N}.
\end{eqnarray}
Since $\bm{y}_{ij}$ is deterministic, $p(\bm{z}-\bm{y}_{ij}|\bm{0},\Sigma_i)$ can be regarded as a Gaussian density function with mean $\bm{y}_{ij}$ and covariance matrix $\Sigma_i$. We use notation $\bm{\mu}_{ij}$ to replace $\bm{y}_{ij}$ and  then $p(\bm{z}-\bm{y}_{ij}|\bm{0},\Sigma_i)$ can be rewritten as $p(\bm{z}|\bm{\mu}_{ij},\Sigma_i)$. Since inequality (\ref{P5-13}) holds for each $\mathscr{I}_i$, $i=1,\cdots,N$, we have
\begin{eqnarray}\label{P5-12}
&&\int_{\mathbb{R}^m}\left|\underline{f}_i\mathbf{1}_{\mathscr{I}_i}(\bm{z})-\sum_{j=1}^{n_i}\underline{f}_i|\mathscr{I}_{ij}|p(\bm{z}|\bm{\mu}_{ij},\Sigma_i)\right|d\bm{z}<\frac{\delta}{2N},~~
i=1,\cdots,N.
\end{eqnarray}

Finally, denote $\pi_{ij}=\underline{f}_i|\mathscr{I}_{ij}|$ and $p^*(\bm{z})=\sum_{i=1}^N\sum_{j=1}^{n_i}\pi_{ij}p(\bm{z}|\bm{\mu}_{ij},\Sigma_i)$. Then 
\begin{eqnarray*}
&&\sum_{i=1}^N\sum_{j=1}^{n_i}\pi_{ij}=\sum_{i=1}^N\underline{f}_i|\mathscr{I}_{i}|=\int_{\mathbb{R}^m}f_N(\bm{z})d\bm{z}\rightarrow1
\end{eqnarray*}
as $N\rightarrow+\infty$. Furthermore, combining (\ref{P5-1}) and (\ref{P5-12}) derives
\begin{eqnarray*}
&&\int_{\mathbb{R}^m}|f(\bm{z})-p^*(\bm{z})|d\bm{z}\\
&&<\int_{\mathbb{R}^m}|f(\bm{z})-f_N(\bm{z})|d\bm{z}+\int_{\mathbb{R}^m}|f_N(\bm{z})-p^*(\bm{z})|d\bm{z}\\
&&<\int_{\mathbb{R}^m}|f(\bm{z})-f_N(\bm{z})|d\bm{z}+\sum_{i=1}^N\int_{\mathbb{R}^m}\left|\underline{f}_i\mathbf{1}_{\mathscr{I}_i}(\bm{z})-\sum_{j=1}^{n_i}\underline{f}_i|\mathscr{I}_{ij}|p(\bm{z}|\bm{\mu}_{ij},\Sigma_i)\right|d\bm{z}\\
&&<\frac{\delta}{2}+\sum_{i=1}^N\frac{\delta}{2N}=\delta
\end{eqnarray*}
The proof is completed.
\end{proof}

Theorem \ref{fitness} shows that condition number constrained GMD still has the fitting capacity. It is worth mentioning that although it is possible to reduce the underlying fitness error and the asymptotic approximation error, the cost is that it may increase the number of components.
\section{Reformulation and Algorithm for QCCP under GMD}
In this section, we develop the solution methodology by reformulating the original problem to a tractable one under the condition that the distribution of $c(\bm{\xi},\bm{x})$ can be well approximated by the associated asymptotic GMD. More specifically, according to Theorem \ref{asy_GMD}, the density function of $c(\bm{\xi},\bm{x})$ is approximated as
\begin{eqnarray*}
&&p(z)=\sum_{i=1}^K\pi_ip\left(z|\mu_i(\bm{x}), \sigma_i^2(\bm{x})\right),
\end{eqnarray*}
where $\mu_i(\bm{x})=E_{p_i}(c(\bm{\xi},\bm{x}))$ and $\sigma_i^2(\bm{x})=V_{p_i}(c(\bm{\xi},\bm{x}))$.

Following \cite{hu2021},  problem ${\rm QCCP}$ can be equivalently formulated as
\begin{eqnarray*}
{\rm QCCP_0}&\min\limits_{\bm{x}\in \mathscr{X}}& \rho(\bm{x})\\
&{\rm s.t.}&\sum_{i=1}^K\pi_i\Phi\left(-\frac{\mu_i(\bm{x})}{\sigma_i(\bm{x})}\right)\geq1-\alpha,
\end{eqnarray*}
where $\Phi(\cdot)$ is the cumulative distribution function of standard normal distribution. By introducing auxiliary vector $\bm y$, the above problem can be further equivalently reformulated as
\begin{eqnarray*}
{\rm QCCP_1}&\min\limits_{\bm{x}\in \mathscr{X},\bm{y}\in\mathscr{Y}}& \rho(\bm{x})\\
&{\rm s.t.}&\Phi^{-1}(y_i)\sigma_i(\bm{x})+\mu_i(\bm{x})\leq0,~ i=1,\cdots,K,\\
&&\sum_{i=1}^K\pi_iy_i\geq1-\alpha,
\end{eqnarray*}
where $\mathscr{Y}=\{\bm{y}\in\mathbb{R}^K:0\leq y_i\leq1,i=1,\cdots,K\}$.

To derive a local optimal solution, some first-order algorithms can be used to solve ${\rm QCCP_1}$. In the following we consider a linear form of $A(\bm{x})$, $\bm{a}(\bm{x})$ and $a(\bm{x})$ and investigate its global optimal solution, i.e.,
\begin{eqnarray}\label{lin_assump}
&&A(\bm{x})=A_0+\sum_{i=1}^nx_iA_i,~\bm{a}(\bm{x})=\bm{a}_0+\sum_{i=1}^nx_i\bm{a}_i~\mbox{and}~a(\bm{x})=a_0+\sum_{i=1}^nx_ia_i.
\end{eqnarray}
Then a lemma from \cite{zhu2020} is introduced to reformulate ${\rm QCCP_1}$.
\begin{lemma}\label{linear_moment}
If $A(\bm{x})$, $\bm{a}(\bm{x})$ and $a(\bm{x})$ satisfy the form of (\ref{lin_assump}), then
\begin{eqnarray*}
E_{p_i}(c(\bm{\xi},\bm{x})) &=& \bm{\nu}_i^{\top}\widehat{\bm{x}},\\
V_{p_i}(c(\bm{\xi},\bm{x})) &=& \widehat{\bm{x}}^{\top}\left(\Psi_i+\frac{1}{2}\Phi_i\right)\widehat{\bm{x}},
\end{eqnarray*}
where
\begin{eqnarray*}
\widehat{\bm{x}}&=&(1;\bm{x}),\\
\bm{\nu}_i&=&\left(\frac{1}{2}{\rm tr}(A_j\Sigma_i)+\frac{1}{2}\bm{\mu}_i^{\top}A_j\bm{\mu}_i+\bm{a}_j^{\top}\bm{\mu}_i+a_j\right)_{j=0,\cdots,n},\\
\Psi_i&=&(A_0\bm{\mu}_i+\bm{a}_0,\cdots,A_n\bm{\mu}_i+\bm{a}_n)^{\top}\Sigma_i(A_0\bm{\mu}_i+\bm{a}_0,\cdots,A_n\bm{\mu}_i+\bm{a}_n),\\
\Phi_i&=&({\rm tr}(A_k\Sigma_iA_l\Sigma_i))_{k,l=0,\cdots,n}.
\end{eqnarray*}
\end{lemma}
According to Lemma \ref{linear_moment}, problem ${\rm QCCP_1}$ is equivalent to
\begin{eqnarray*}
{\rm QCCP_{{\rm lin}}}&\min\limits_{\bm{x}\in \mathscr{X},\bm{y}\in\mathscr{Y}}& \rho(\bm{x})\\
&{\rm s.t.}&\Phi^{-1}\left(y_i\right)\sqrt{\widehat{\bm{x}}^{\top}\left(\Psi_i+\frac{1}{2}\Phi_i\right)\widehat{\bm{x}}}+\bm{\nu}_i^{\top}\widehat{\bm{x}}\leq0,~ i=1,\cdots,K,\\
&&\sum_{i=1}^K\pi_iy_i\geq1-\alpha.
\end{eqnarray*}
In general, problem ${\rm QCCP_{lin}}$ is a non-convex program. For problems with this structure, \cite{hu2021} propose a spatial branch-and-bound (BB) algorithm to derive the global optimal solution which is placed in Appendix D. Roughly speaking, the BB algorithm relaxes the constraint over the divided sub-domain of cube $\mathscr{Y}$ to solve an easier problem and generate a  lower bound for the sub-domain in each iteration. Once the optimal value of the relaxed problem is smaller than the current global optimal value and the optimal solution to the relaxed problem is also feasible to the original problem, the global solution of the original problem is updated and those branches with local lower bounds that are larger than the global one will be pruned. This iterative process generates a sequence of non-increasing values that converges to the global optimal value by solving a sequence of relaxed problems on the subdivided cubes, which is discussed in detail in the sequel. 

Suppose that $\widetilde{\mathscr{Y}}=\{\bm{y}\in\mathbb{R}^K:l_i\leq y_i\leq u_i, i=1,\cdots,K\}$ is a subdivided cube. A relaxed problem shown in the following is solved to derive a lower bound for the primary problem ${\rm QCCP_{lin}}$ over $\widetilde{\mathscr{Y}}$.
\begin{eqnarray*}
{\rm QCCP_{relax}}&\min\limits_{\bm{x}\in\mathscr{X},\bm{y}\in\widetilde{\mathscr{Y}}}& \rho(\bm{x})\\
&s.t.&\Phi^{-1}\left(l_i\right)\sqrt{\widehat{\bm{x}}^{\top}\left(\Psi_i+\frac{1}{2}\Phi_i\right)\widehat{\bm{x}}}+\bm{\nu}^{\top}_i\widehat{\bm{x}}\leq 0,~i=1,\cdots,K,\\
&&\sum_{i=1}^K\pi_iy_i\geq1-\alpha.
\end{eqnarray*}
\cite{hu2021} point out that if $\pi_i\geq2\alpha$ holds for $i=1,\cdots,K$, then the ${\rm QCCP_{relax}}$ is a convex program. 

In the real implementation, we can only achieve an $\epsilon$-optimal solution with the algorithm due to the termination criterion in the code. Suppose $(\bm{x},\bm{y})$ is a feasible solution to problem ${\rm QCCP_{relax}}$. If $h_i(\bm{x},\bm{y})-h_i^r(\bm{x},\bm{y})\leq\epsilon$, $i=1,\cdots,K$, where
\begin{eqnarray*}
&&h_i(\bm{x},\bm{y})=\Phi^{-1}\left(y_i\right)\sqrt{\widehat{\bm{x}}^{\top}\left(\Psi_i+\frac{1}{2}\Phi_i\right)\widehat{\bm{x}}}+\bm{\nu}^{\top}_i\widehat{\bm{x}},\\
&&h_i^r(\bm{x},\bm{y})=\Phi^{-1}\left(l_i\right)\sqrt{\widehat{\bm{x}}^{\top}\left(\Psi_i+\frac{1}{2}\Phi_i\right)\widehat{\bm{x}}}+\bm{\nu}^{\top}_i\widehat{\bm{x}}, 
\end{eqnarray*}
then we call $(\bm{x},\bm{y})$ an $\epsilon$-feasible solution to ${\rm QCCP_{lin}}$. If $(\bm{x},\bm{y})$ is the optimal solution and satisfy $h_i(\bm{x},\bm{y})-h_i^r(\bm{x},\bm{y})\leq\epsilon$, $i=1,\cdots,K$, we call it $\epsilon$-optimal solution. It is worth mentioning that the if the solution $(\bm{x},\bm{y})$ is $\epsilon$-feasible, it indicates that $0<y_i<1$, $i=1,\cdots,K$. Otherwise, if $y_i=1$, $h_i(\bm{x},\bm{y})-h_i^r(\bm{x},\bm{y})=+\infty$, which contradicts the condition of $\epsilon$-feasible. If $y_i=0$, both $h_i(\bm{x},\bm{y})$ and $h^r_i(\bm{x},\bm{y})$ are equal to $-\infty$. It remains inconvenient to deal with the subtraction of two infinities in implementation. Therefore, in the real implementation, we replace the feasible set $\mathscr{Y}$ with an approximated set $\mathscr{Y}^0=\{\bm{y}\in\mathbb{R}^K:\underline{y}_i\leq y_i\leq\overline{y}_i,i=1,\cdots,K\}$  as the input of the algorithm, where $0<\underline{y}_i\leq\overline{y}_i<1$, and $|\underline{y}_i|$ and $|1-\overline{y}_i|$ are sufficiently small numbers.

\begin{proposition}\label{complex}
Assume that $\mathscr{X}$ is compact and the initial cube $\mathscr{Y}^0$ of the Algorithm 1 satisfies $0<\underline{y}_i\leq\overline{y}_i<1$, $i=1,\cdots,K$. To generate an $\epsilon$-optimal solution to ${\rm QCCP_{lin}}$ or verify its infeasibility, the number of relaxed sub-problem ${\rm QCCP_{relax}}$ needed to be solved is at most
\begin{eqnarray*}
&&\left\lfloor\frac{x_{\max}d^*}{\epsilon}[y_{k^*}]\sqrt{\lambda^{*}}\right\rfloor^K,
\end{eqnarray*}
where $\lfloor\cdot\rfloor$ is the maximum integer smaller than or equal to the given number, $x_{\max}=\max\limits_{\bm{x}\in \mathscr{X}}\left\{||\widehat{\bm{x}}||\right\}$, $[y_{k^*}]=\max\{\overline{y}_i-\underline{y}_i,i=1,\cdots,K\}$, $d^*=\max\left\{\frac{1}{\phi\left(\Phi^{-1}(\underline{y}^*)\right)},\frac{1}{\phi\left(\Phi^{-1}(\overline{y}^*)\right)}\right\}$, $\phi(\cdot)$ is the standard normal density function,
$\underline{y}^*=\min\left\{\underline{y}_i,i=1,\cdots,K\right\}$, $\overline{y}^*=\max\left\{\overline{y}_i,i=1,\cdots,K\right\}$, and $\lambda^*$ is the maximum eigenvalue among $\Psi_i+\frac{1}{2}\Phi_i$, $i=1,\cdots,K$.
\end{proposition}
\begin{proof}
For any $\bm{x}\in\mathscr{X}$ and $\bm{y}\in\mathscr{Y}^0$, 
\begin{eqnarray*}
h_i(\bm{x},\bm{y})-h^r_i(\bm{x},\bm{y})&=&\left(\Phi^{-1}(y_i)-\Phi^{-1}(\underline{y}_i)\right)\sqrt{\widehat{\bm{x}}^{\top}\left(\Psi_i+\frac{1}{2}\Phi_i\right)\widehat{\bm{x}}}\\
&=&(y_i-\underline{y}_i)\left.\frac{d}{dy}\Phi^{-1}(y)\right|_{y=y_i^0}\sqrt{\widehat{\bm{x}}^{\top}\left(\Psi_i+\frac{1}{2}\Phi_i\right)\widehat{\bm{x}}}\\
&\leq&(\overline{y}_i-\underline{y}_i)d^*x_{\max}\sqrt{\lambda^{*}},
\end{eqnarray*}
where $y_i^0\in[\underline{y}_i,\overline{y_i}]$. The last inequality hold since the derivative of $\Phi^{-1}(y)$ is $\frac{1}{\phi\left(\Phi^{-1}(y)\right)}$ and $\lambda^*I-(\Psi_i+\frac{1}{2}\Phi_i)$ is positive semi-definite and thus $\lambda^*x_{\max}^2\geq\lambda^*||\widehat{\bm{x}}||^2\geq\widehat{\bm{x}}^{\top}\left(\Psi_i+\frac{1}{2}\Phi_i\right)\widehat{\bm{x}}$. 

Suppose that a cube $\widetilde{\mathscr{Y}}=\prod_{i=1}^{K}[l_i,u_i]$ satisfies \begin{eqnarray*}
&&u_i-l_i\leq\frac{\epsilon}{d^*x_{\max}\sqrt{\lambda^{*}}},~i=1,\cdots,K.
\end{eqnarray*}
If problem ${\rm QCCP_{relax}}$ is feasible over the $\widetilde{\mathscr{Y}}$, denote the optimal solution as $(\widetilde{\bm{x}}^*,\widetilde{\bm{y}}^*)$. Since $h_i(\widetilde{\bm{x}}^*,\widetilde{\bm{y}}^*)-h^r_i(\widetilde{\bm{x}}^*,\widetilde{\bm{y}}^*)\leq\epsilon$, The solution is an $\epsilon$-optimal solution over $\widetilde{\mathscr{Y}}$. If problem ${\rm QCCP_{relax}}$ is infeasible over $\widetilde{\mathscr{Y}}$, problem ${\rm QCCP_{lin}}$ must be infeasible over $\widetilde{\mathscr{Y}}$. In both cases, $\widetilde{\mathscr{Y}}$ will not be divided into two smaller cubes. 

According to the algorithm, the initial cube $\mathscr{Y}^0$ will be divided into two smaller cubes from the middle point of the longest edge if the longest edge is larger than $\frac{\epsilon}{d^*x_{\max}\sqrt{\lambda^{*}}}$. Now consider the $i$th edge of the cube with the length of $[y_i]=\overline{y}_i-\underline{y}_i\gg\frac{\epsilon}{d^*x_{\max}\sqrt{\lambda^{*}}}$. According to the rule cutting in the middle, to generate all instances with the length of this edge equal to $\left(\frac{1}{2}\right)^1[y_i]$, it is only required $2^0=1$ time of division. If all instances on this edge are equal to $\left(\frac{1}{2}\right)^2[y_i]$, it is required $2^0+2^1=3$ times of division. It is not difficult to conclude that $2^{n}-1$ times of division can obtain all instances that the length of this edge is equal to $\left(\frac{1}{2}\right)^n[y_i]$.

Let $n^*={\rm log}_2\left\lfloor\frac{d^*x_{\max}\sqrt{\lambda^{*}}[y_i]}{\epsilon}+1\right\rfloor$. Then the instances with the length of the $i$th edge are equal to
\begin{eqnarray*}
&&[y_i]\left(\frac{1}{2}\right)^{{\rm log}_2\left\lfloor\frac{d^*x_{\max}\sqrt{\lambda^{*}}[y_i]}{\epsilon}+1\right\rfloor}=[y_i]\frac{1}{\left\lfloor\frac{d^*x_{\max}\sqrt{\lambda^{*}}[y_i]}{\epsilon}+1\right\rfloor}<[y_i]\frac{1}{\frac{d^*x_{\max}\sqrt{\lambda^{*}}[y_i]}{\epsilon}}=\frac{\epsilon}{d^*x_{\max}\sqrt{\lambda^{*}}}.
\end{eqnarray*}
Therefore this edge is no longer divided and the corresponding number of partition is 
\begin{eqnarray*}
&&2^{n^*}-1=\left\lfloor\frac{d^*x_{\max}\sqrt{\lambda^{*}}[y_i]}{\epsilon}+1\right\rfloor-1=\left\lfloor\frac{d^*x_{\max}\sqrt{\lambda^{*}}[y_i]}{\epsilon}\right\rfloor.
\end{eqnarray*}

Note that each division corresponds to one computation of problem ${\rm QCCP_{relax}}$. This implies that the number of division is equal to the number of computation of ${\rm QCCP_{relax}}$. Therefore, the number of problem ${\rm QCCP_{relax}}$ solved is at most
\begin{eqnarray*}
&&\left\lfloor\frac{d^*x_{\max}\sqrt{\lambda^{*}}[y_1]}{\epsilon}\right\rfloor\times\cdots\times\left\lfloor\frac{d^*x_{\max}\sqrt{\lambda^{*}}[y_K]}{\epsilon}\right\rfloor\leq
\left\lfloor\frac{d^*x_{\max}\sqrt{\lambda^{*}}[y_{k^*}]}{\epsilon}\right\rfloor^K.
\end{eqnarray*}
The proof is completed.
\end{proof}
Proposition \ref{complex} provides a worst-case estimation of the number of sub-problems $\rm QCCP_{lin}$ solved in the BB algorithm. The complexity of the algorithm is exponential with respect to the number of components $K$. Fortunately, $K$ is usually a small number. Furthermore, as only the sub-cubes that are interacted with the hyperplane $\sum_{i=1}^K\pi_iy_i=1-\alpha$ are considered in the algorithm, the number of sub-problems needed to be solved is far fewer than the worst-case estimation in proposition \ref{complex}.

\section{Numerical Simulation}
In this section we first conduct numerical experiments to verify the theoretical results. Then we test the efficiency of the BB algorithm.

\subsection{Test of Conditions for Asymptotic Distribution}
We start from equation (\ref{sum_chi}) to test how the mean $\alpha_i$'s and the weight $\omega_i$'s affect the asymptotic errors by directly comparing the density function of $Z_h$ with the corresponding theoretical asymptotic density function. 

We simulate 20000 samples from $h=10$ independent normal distributed variables with unit variance and different mean settings, and then calculate the weighted sum of the square of these samples to obtain samples of $Z_h$. The sample histogram of $Z_h$ is approximated by a smooth curve, which is a representation of the density function of $Z_h$. The asymptotic distribution is directly derived from Theorem \ref{asy_GMD} by setting $K=1$. In Figure 1, three sub-figures are displayed with different settings of mean and weight: $\alpha_i=0$, $i=1,\cdots,h$ and $\omega_i=\frac{1}{2^i}$, $i=1,\cdots,h$ for the left sub-figure; $\alpha_i=0$, $i=1,\cdots,h$ and $\omega_i=1$, $i=1,\cdots,m$ for the middle sub-figure; $\alpha_i$, $i=1,\cdots,h$ are uniformly sampled from (0,10) and $\omega_i=\frac{1}{2^i}$, $i=1,\cdots,h$ for the right sub-figure. As we can see, when the means are non-zeros, the asymptotic approximation error is small even though the weights of all random variables are not even.  When the weights are equal, the asymptotic approximation error decreases as compared with the unequal case. These findings are consistent with the result of Theorem \ref{con_rate}.

\begin{figure}[H]
  \begin{minipage}{0.3\linewidth}
  \centerline{\includegraphics[width=6cm]{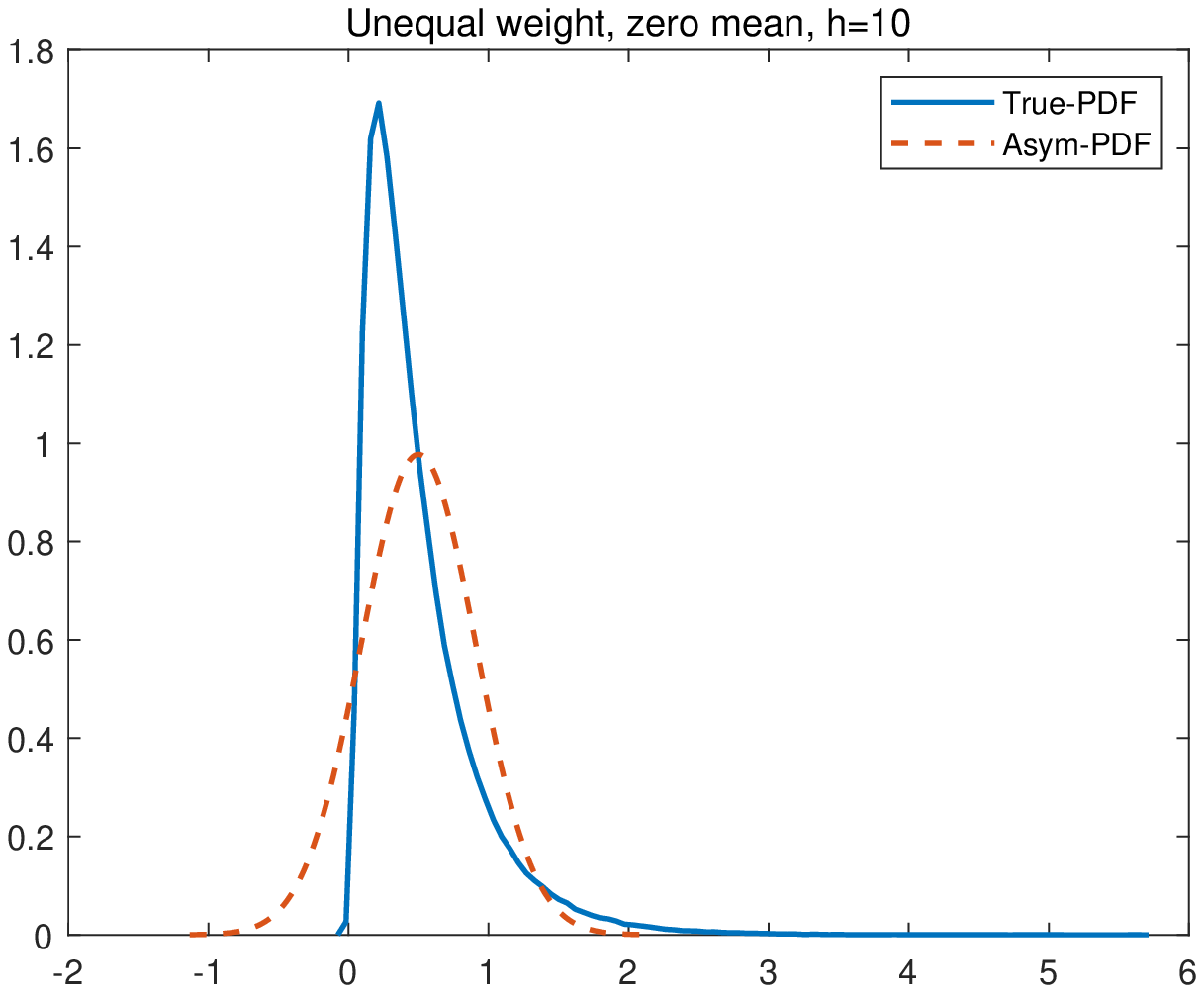}}
  \end{minipage}
  \hfill
  \begin{minipage}{0.3\linewidth}
  \centerline{\includegraphics[width=6cm]{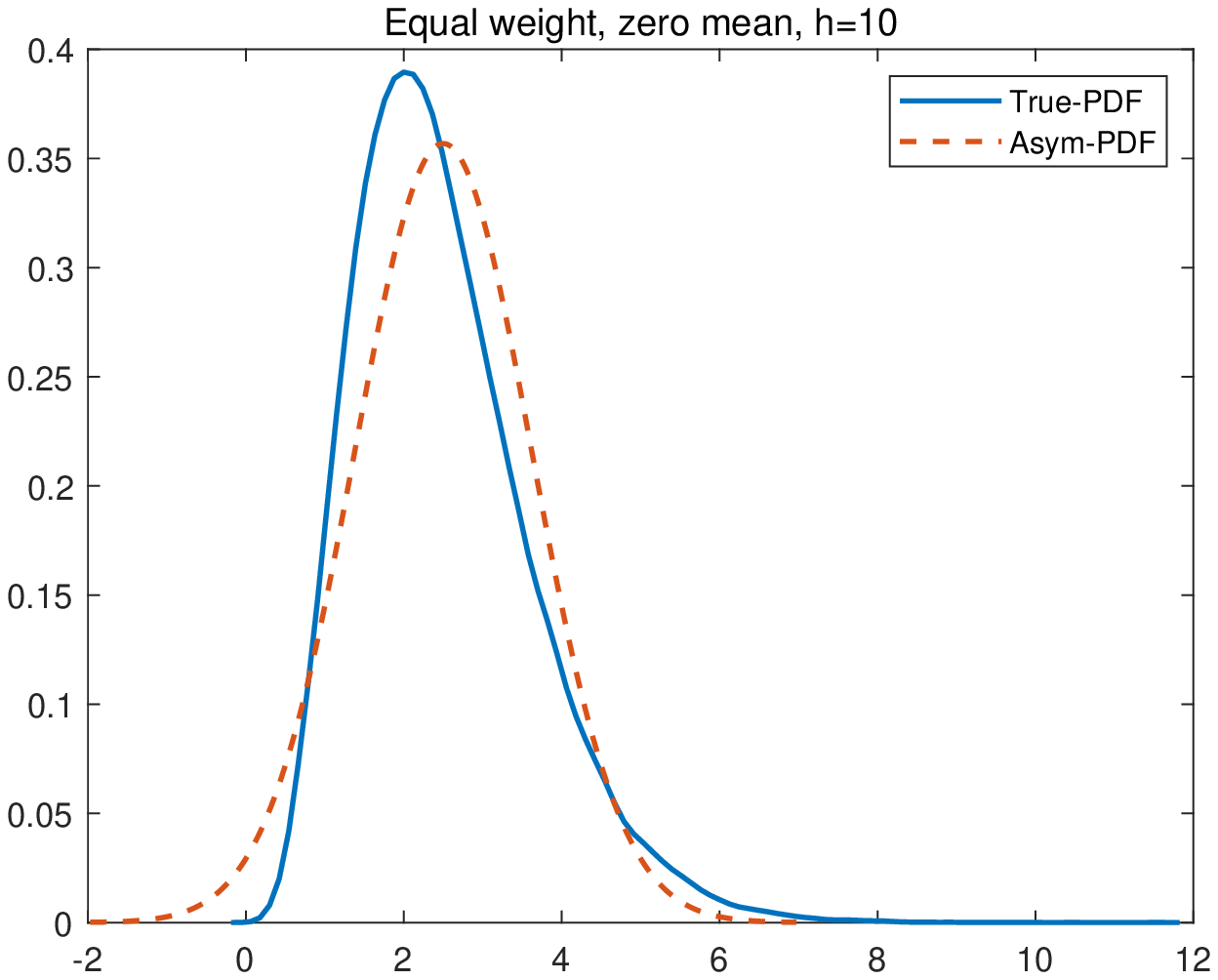}}
  \end{minipage}
  \hfill
  \begin{minipage}{0.3\linewidth}
  \centerline{\includegraphics[width=6cm]{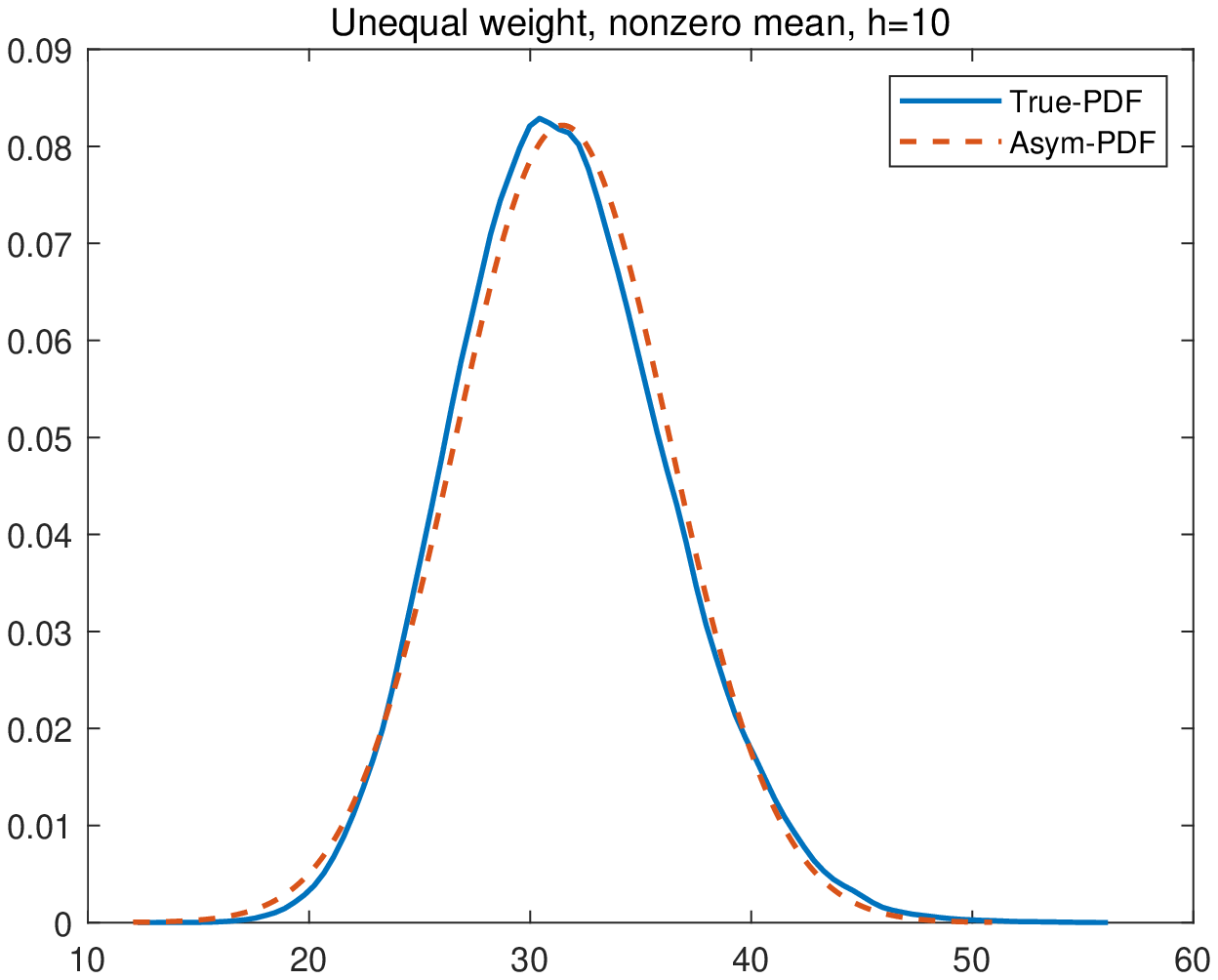}}
  \end{minipage}
  \caption{Effects of different factors on asymptotic approximation errors (GMD with $K=1$)}
\end{figure}

Then we use other numerical experiments to directly verify the result of Theorem \ref{asy_GMD}, again by comparing the density functions. We set $K=3$, $A={\rm I}$, $\bm{a}=\bm{0}$ and $a=0$. The covariance matrix of each component is generated by $\Sigma_i=D\Lambda D^{\top}$, where $D$ is an orthogonal matrix that is randomly generated, and $\Lambda$ is a diagonal matrix whose elements are uniformly generated from (0,1). The mixture weights $\pi_i$, $i=1,\cdots,K$, are uniformly generated and then regularized. 

The results are displayed in Figure 2. By comparing the left and the middle sub-figures, we can see that as $m$ increases, the true distribution of $c(\bm{\xi},\bm{x})$ converges to its asymptotic distribution. In the right sub-figure, we set the means of GMD components to be nonzero. Specifically, each $\bm{\mu}_i$ is uniformly generated from (0,10). It shows that the asymptotic approximation error of the case with nonzero means is smaller than that of the case with zero means.

\begin{figure}[H]
  \begin{minipage}{0.3\linewidth}
  \centerline{\includegraphics[width=6cm]{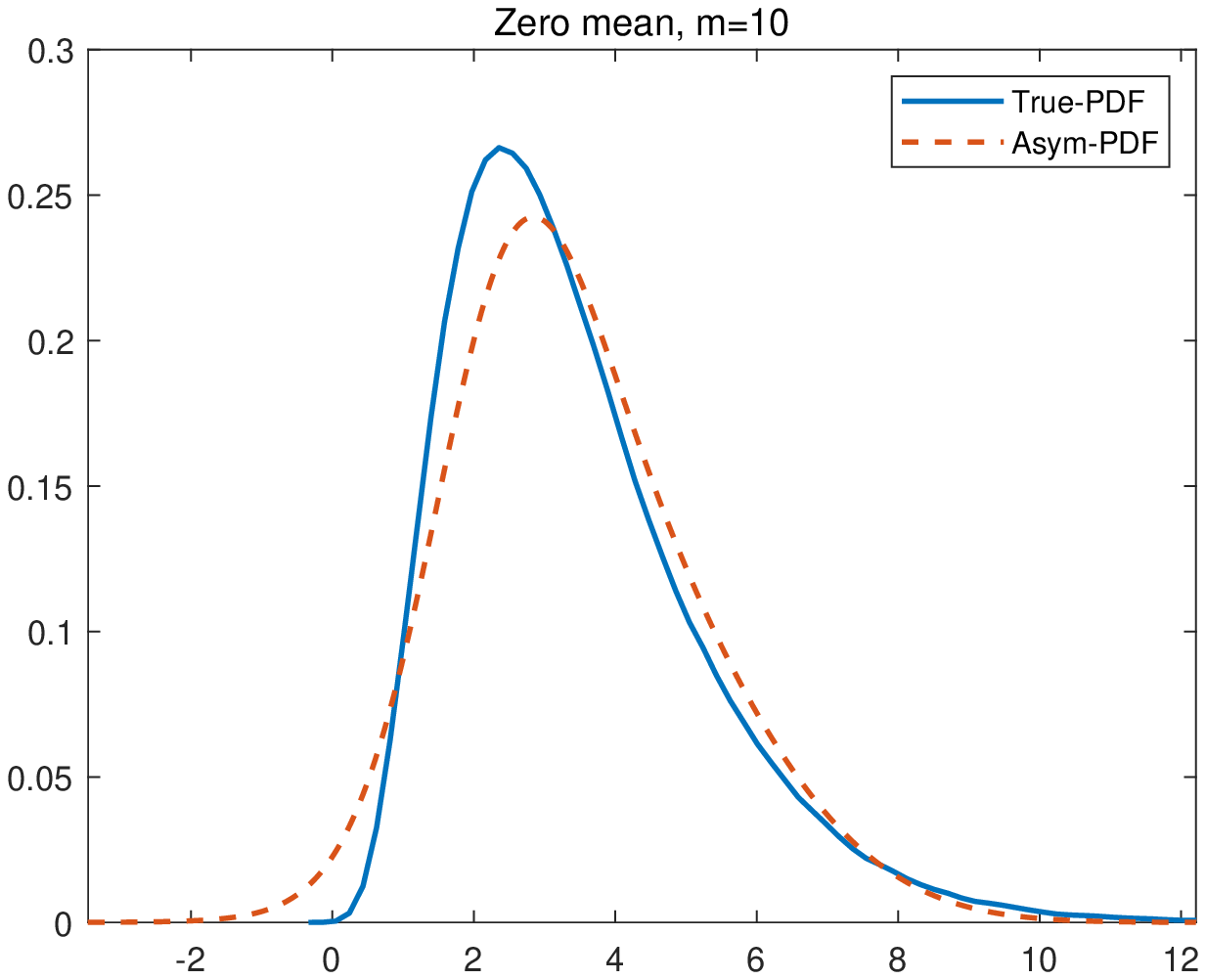}}
  \end{minipage}
  \hfill
  \begin{minipage}{0.3\linewidth}
  \centerline{\includegraphics[width=6cm]{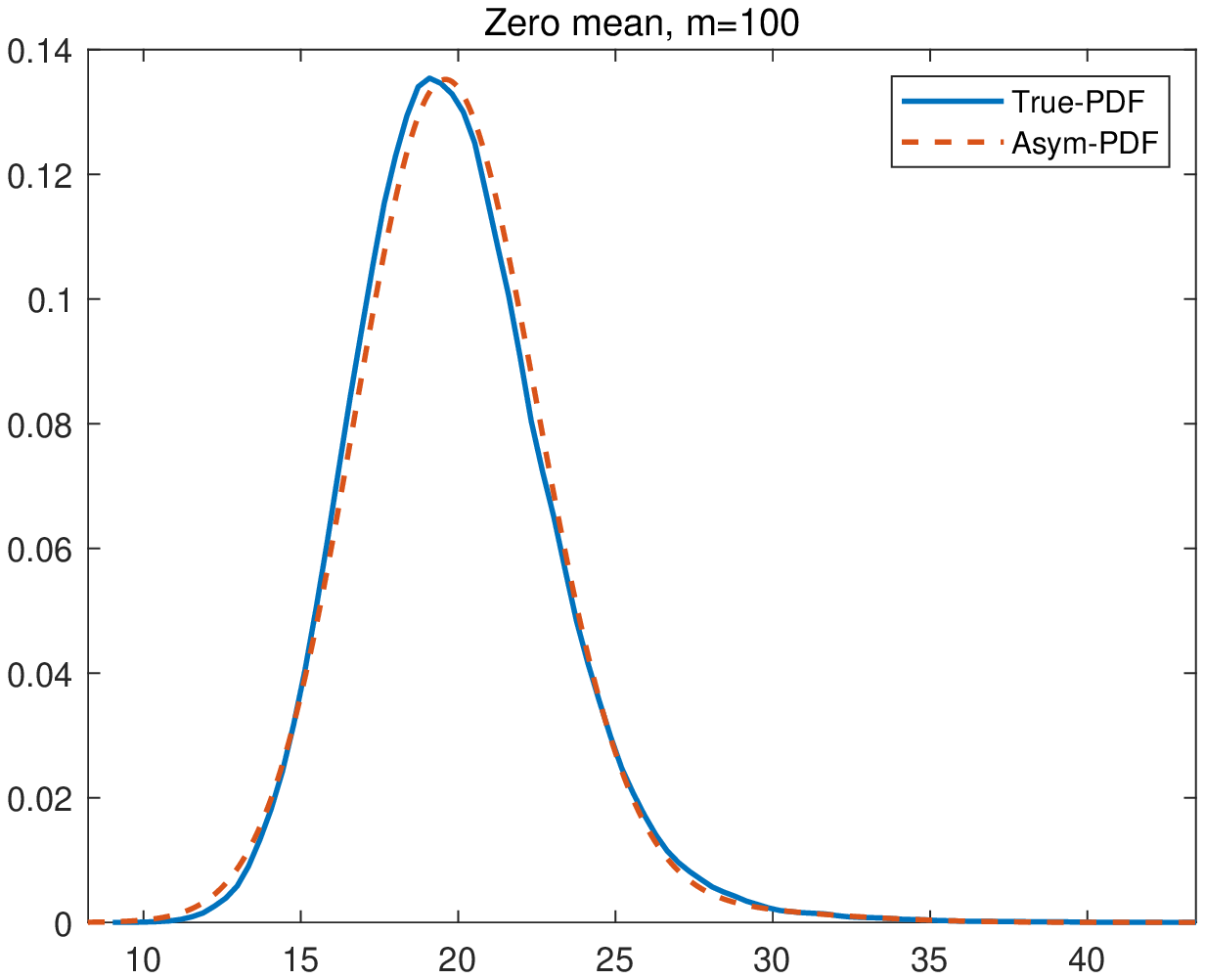}}
  \end{minipage}
  \hfill
  \begin{minipage}{0.3\linewidth}
  \centerline{\includegraphics[width=6cm]{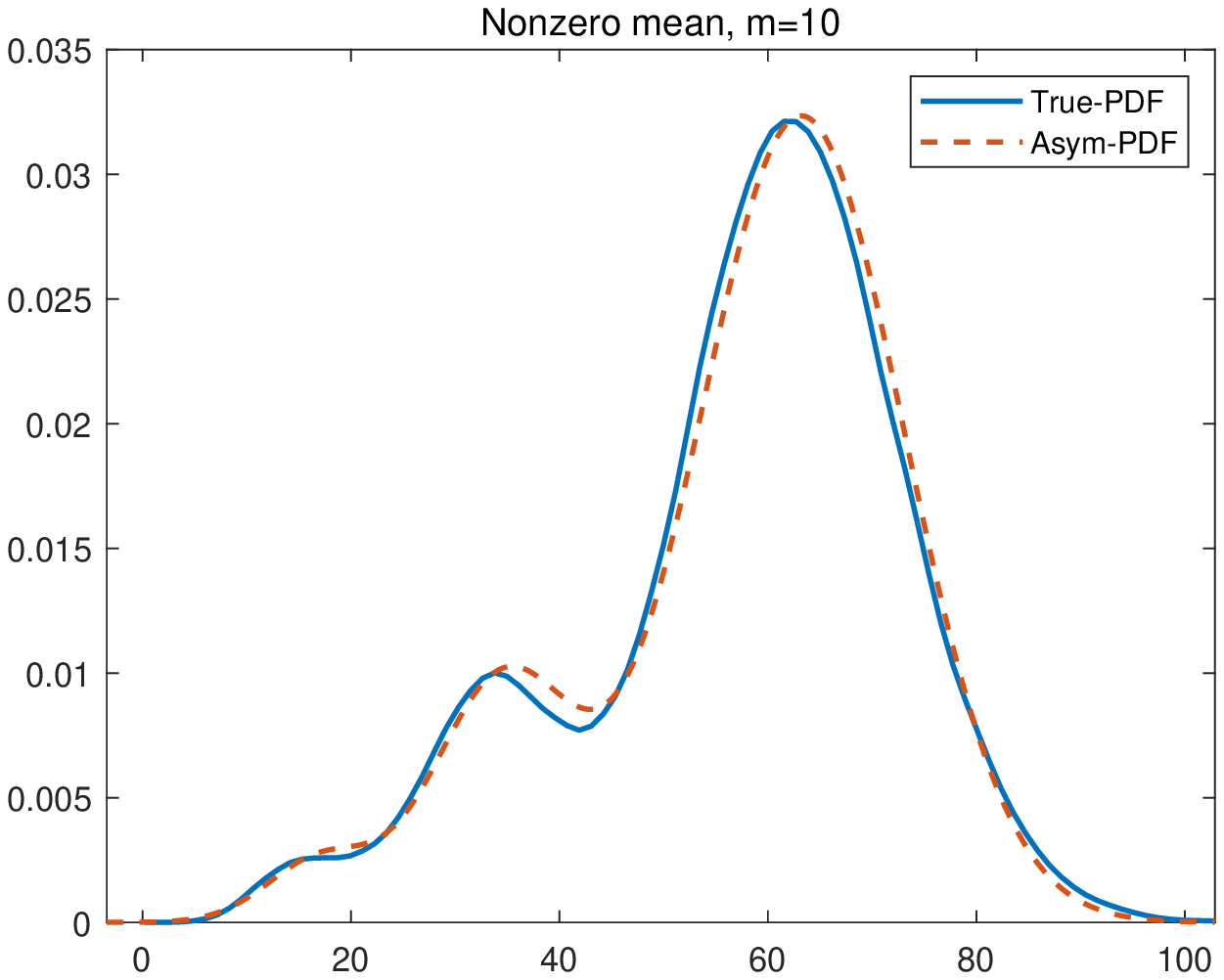}}
  \end{minipage}
  \caption{Effects of different factors on asymptotic approximation errors (GMD with $K=3$)}
  \end{figure}

Finally, we provide simulations to verify the implications of Theorem \ref{con_rate}, Proposition \ref{con_num} and Theorem \ref{fitness}, which indicate that the restriction of the condition number of $\Sigma_i$, $i=1,\cdots,K$ can reduce the underlying fitness error and asymptotic approximation error simultaneously. The simulation procedures are as follows. First, randomly generate 20000 samples from a 30-dimensional random vector $\bm{\xi}$ that follows a specified distribution as the benchmark for comparison, which is usually unknown in reality. Then, estimate a GMD or a GMD with condition number constraints with the samples. Denote by $\hat{\bm{\xi}}$ the random vector following the estimated GMD. Next, use the Monte Carlo method to draw the density function of $c(\bm{\xi},\bm{x})$ and the density function of $c(\hat{\bm{\xi}},\bm{x})$. By comparing these two density functions we can check the underlying fitness error. Finally, draw the asymptotic density function of $c(\hat{\bm{\xi}},\bm{x})$ according to Theorem \ref{asy_GMD}, and evaluate the asymptotic approximation error by comparing it with the density function of $c(\hat{\bm{\xi}},\bm{x})$.

More specifically, in the first step, we specify the distribution of $\bm{\xi}$ as a GMD with the number of components $K=2$. For the $i$th component, $\bm{\mu}_i$ is uniformly generated from (0,1), $\Sigma_i=LL^{\top}$, where the elements of $L$ are uniformly generated from (0.5,1.5), and the mixture weight $\pi_i$ is randomly generated. In the simulating results shown in Figure 3, the condition number of the covariance matrices are 436510 and 1742000, respectively. In the second step, the EM algorithm that can be referred to \cite{mclachlan1997} is used to estimate the parameters of GMD. In addition, $A={\rm I}$, $\bm{a}$ is randomly generated from (-100,100) and $a=0$.

In Figure 3, ``True-PDF'', ``Esti-PDF'', and ``Asym-PDF'' represent the density function of $c(\bm{\xi},\bm{x})$, $c(\hat{\bm{\xi}},\bm{x})$, and the theoretical asymptotic distribution of $c(\hat{\bm{\xi}},\bm{x})$, respectively. In the left sub-figure, we use the unconstrained GMD in the second step. The estimated density with two components can fit the true density well but the asymptotic density is far away from the estimated density. It implies that the underlying fitness error is small but the asymptotic approximation error is large. When we impose constraints on the condition number, as shown in the middle sub-figure, the estimated density is quite different from the true density but is close to its asymptotic density, indicating that the underlying fitness error is large but the asymptotic approximation error is small. The right sub-figure shows that as the number of components increases, the asymptotic distribution based on the GMD with constraints on condition number can approximate well to the true distribution, implying that the underlying fitness error and asymptotic approximation error simultaneously decrease.
\begin{figure}[H]
  \begin{minipage}{0.3\linewidth}
  \centerline{\includegraphics[width=6cm]{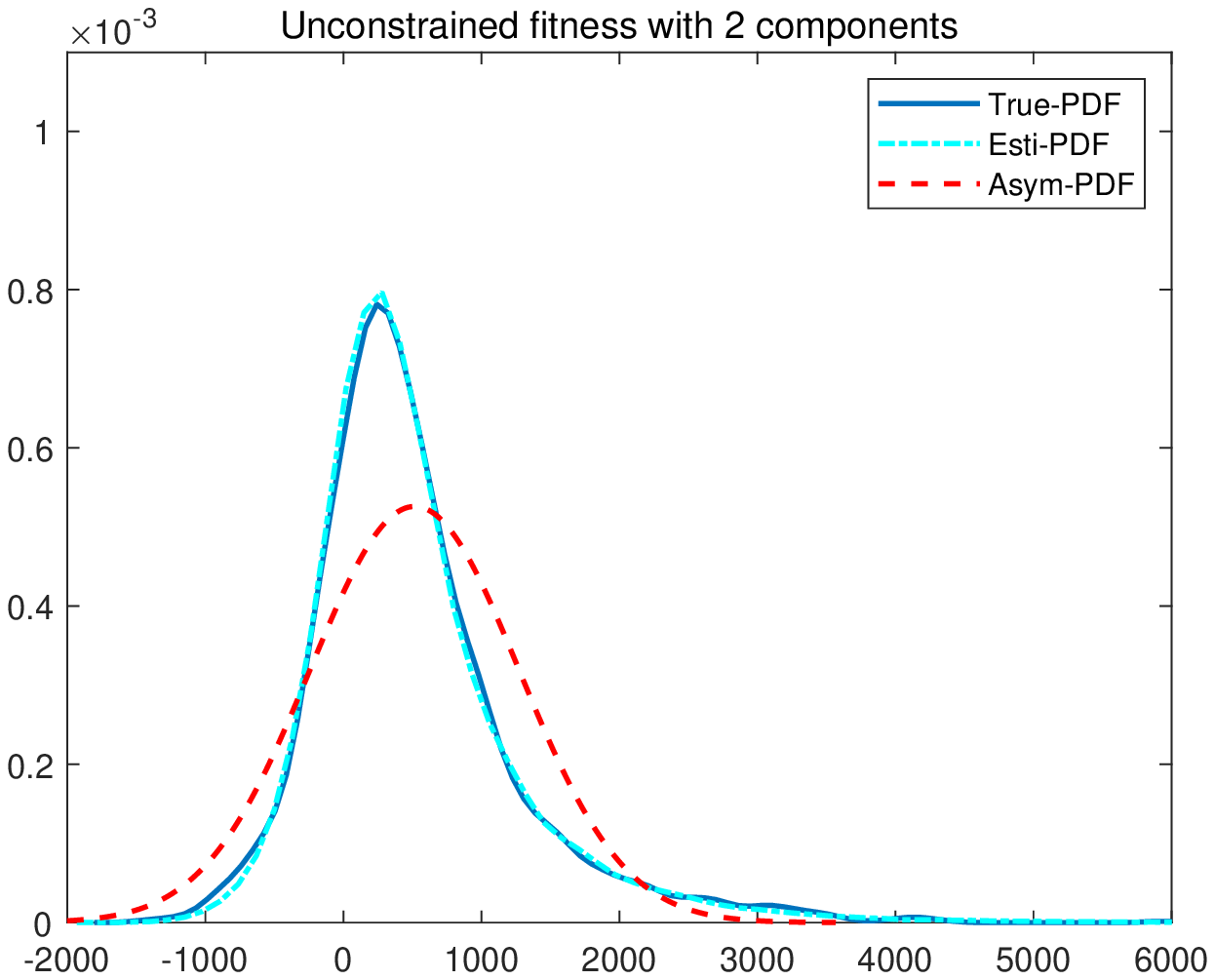}}
  \end{minipage}
  \hfill
  \begin{minipage}{0.3\linewidth}
  \centerline{\includegraphics[width=6cm]{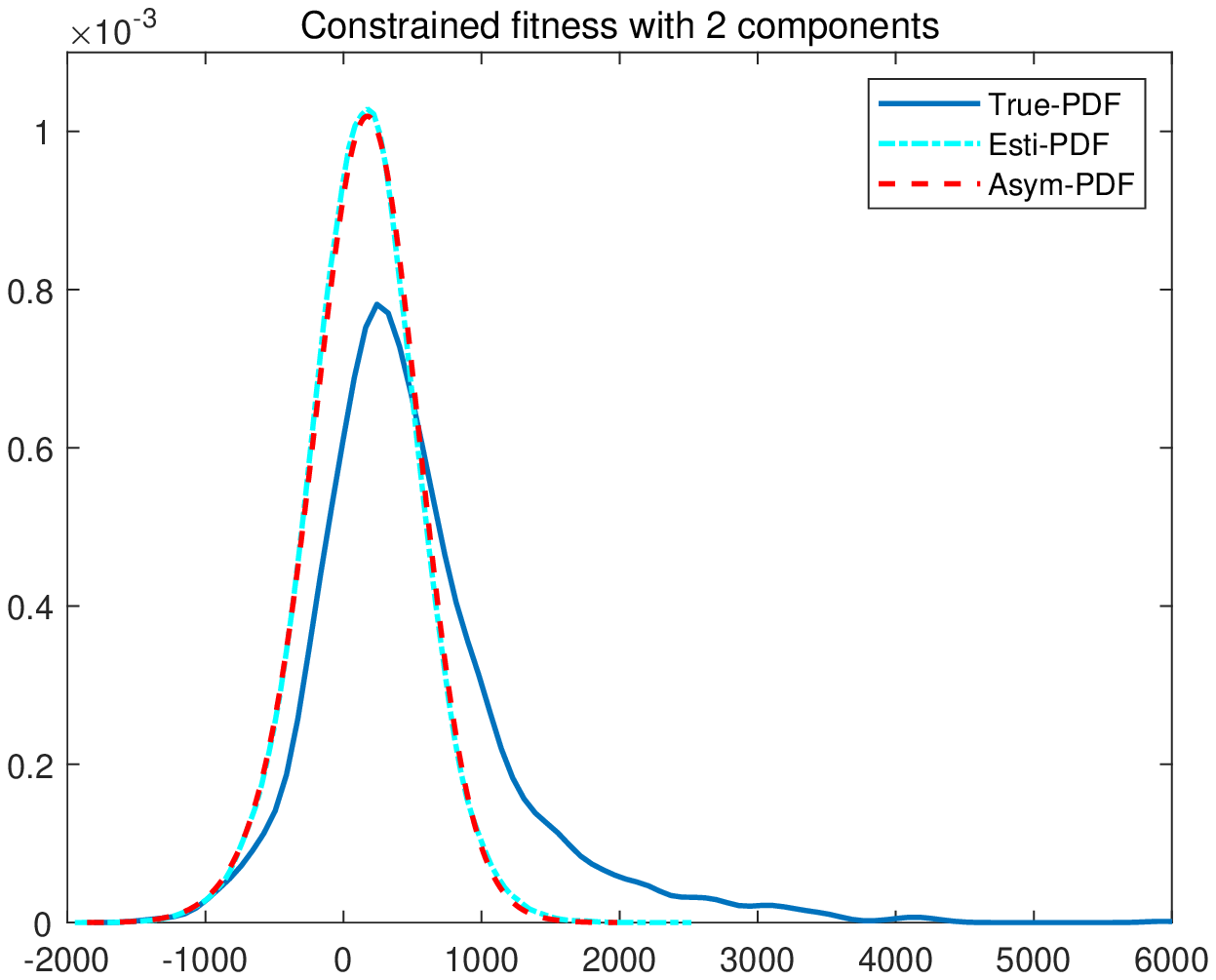}}
  \end{minipage}
  \hfill
  \begin{minipage}{0.3\linewidth}
  \centerline{\includegraphics[width=6cm]{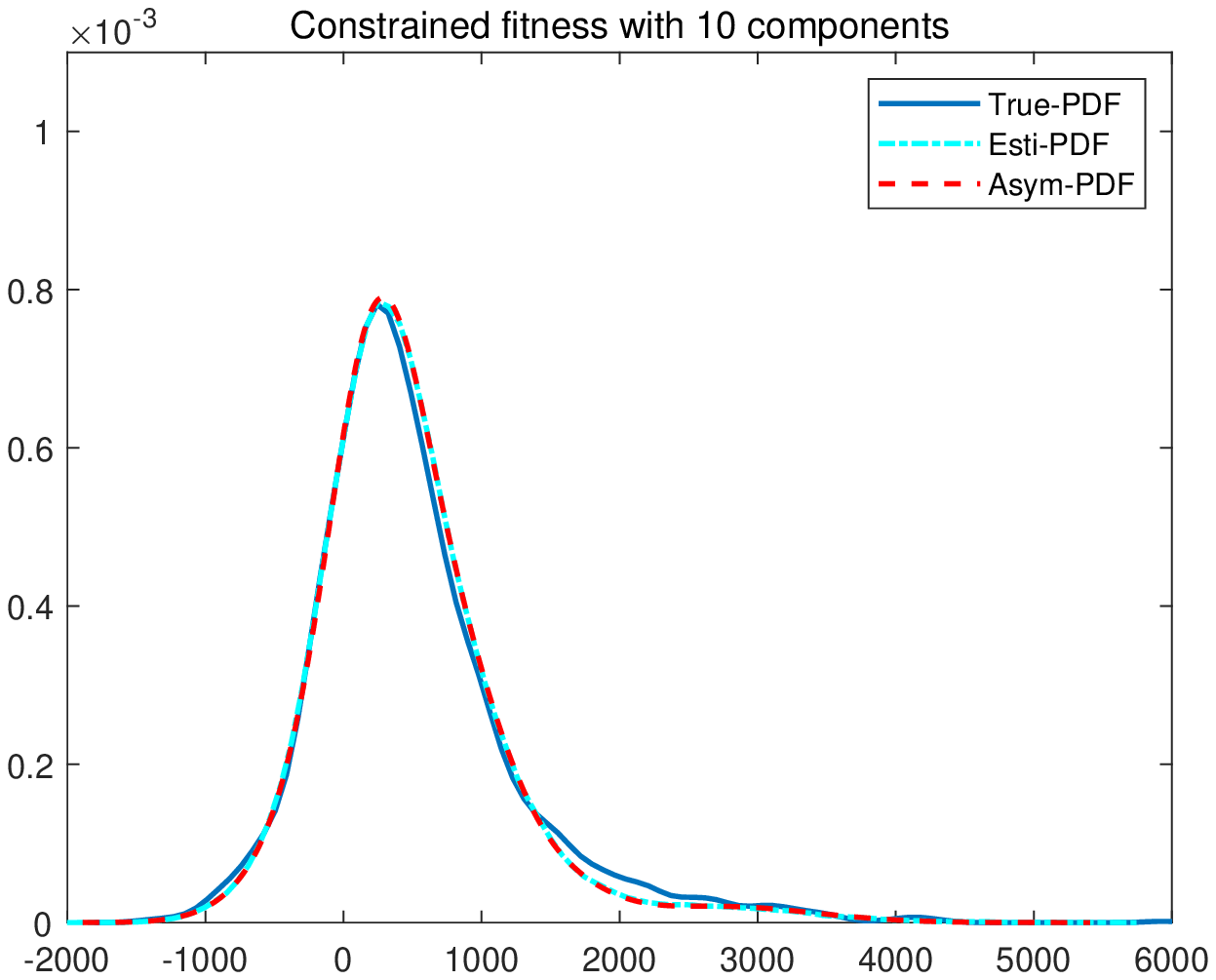}}
  \end{minipage}
  \caption{Fitting capability of condition number constrained GMD}
\end{figure}

\subsection{Test of BB Algorithm}
In this subsection,  we use some numerical examples to test the efficiency of the BB algorithm. We set the objective function as a linear function $\rho(\bm{x})=b_0+\sum_{i=1}^nb_ix_i$, where $b_i$, $i=0,\cdots,n$, are uniformly generated from (-5,5). To construct a non-convex ${\rm QCCP_{lin}}$, for the $i$th component, the entries of $\bm{\mu}_i$ are randomly generated from $(9i-8,9i+1)$, and the covariance matrix is calculated by $LL^{\top}$, where the entries of $L$ are randomly generated from $(i,i+1)$. In addition, we set $m=n$, $A_i=O$, $i=1,\cdots,n$, the $i$th entry of $\bm{a}_i$ is -1 and otherwise 0, $i=1,\cdots,n$, $\bm{a}_0=-\bm{1}$, and $a_i=0$, $i=0,\cdots,n$. For the decision vector $\bm{x}$, let $x_1=1$ and other entries vary from [-100,100]. The confidence level is set at $\alpha=0.05$.

We compare the efficiency of the BB algorithm and the optimization function ``fmincon'' from MATLAB. In the options of ``fmincon'', the algorithm is set as ``SQP'' and the maximal function evaluation and maximal iteration are both set at 30000. For the BB algorithm, if the relative error between the lower bound and the current optimal value is smaller than $10^{-2}$, the algorithm is terminated. If a branch with lower bound $\bm{l}$ and upper bound $\bm{u}$ for $\bm{y}$ satisfies $\|\bm{u}-\bm{l}\|/\|\bm{l}\|\leq10^{-2}$, this branch will not be further divided. We use the optimal solution from ``fmincon'' as the initial point of the BB algorithm. 

Table 1 and Table 2 display the optimal value and computational time of simulated problems with different problem size settings. In this simulation, the mixture weights are randomly generated and satisfy $\pi_i\geq\alpha$, $i=1,\cdots,K$ so that the sub-problems are convex. For each problem size the experiments are executed 5 times and ``Max'', ``Ave'', and ``Min'' represent the maximum, average, and minimum values in the five experiments, respectively. Since the sub-problems of the BB algorithm are all convex, the BB algorithm can obtain the global optimal solution in a short time. In addition, as we can see, compared with the BB algorithm, we find that ``fmincon'' spend less time but only obtain a local optimal solution.

\begin{table}[H]
  \begin{center}
  \caption{Optimal value obtained by BB and Fmincon: convex sub-problem}
  \renewcommand\arraystretch{0.8}
  \begin{tabular*}{\textwidth}{@{\extracolsep{\fill}}ccd{6.3}d{6.3}d{6.3}d{6.3}d{6.3}d{6.3}@{\extracolsep{\fill}}}
  \toprule
   ~ & ~  & \multicolumn{3}{@{}c@{}}{BB} & \multicolumn{3}{@{}c@{}}{Fmincon}\\
  \cmidrule{3-5}\cmidrule{6-8}
  \multicolumn{1}{c}{$K$} &  \multicolumn{1}{c}{$n$} &    \multicolumn{1}{c}{Max} &    \multicolumn{1}{c}{Ave} &   \multicolumn{1}{c}{Min}     & \multicolumn{1}{c}{Max}  &  \multicolumn{1}{c}{Ave} &   \multicolumn{1}{c}{Min} \\
  \midrule
   & 50 & -611.03& 	-3341.81& 	-6345.70& 	-407.19& 	-2962.81& 	-5694.38\\
  2 & 100 &-4493.19& 	-8268.63& 	-12758.77& 	-3920.77& 	-7743.47& 	-12193.84\\
   & 200 &-12522.61& 	-18967.46& 	-35916.11& 	-10690.08& 	-17122.09& 	-33868.77\\
  \midrule
   & 50 & -989.28& 	-2324.47& 	-4538.31& 	-72.49& 	-1618.60& 	-4152.89\\
  3 & 100 &-7722.69& 	-12366.06& 	-17369.61& 	-6550.45& 	-10956.14& 	-16073.30\\
   & 200 &-8932.67& 	-17104.38&	-33098.43& 	-5949.81& 	-14680.56& 	-30000.16\\
  \midrule
   & 50 &-2980.62& 	-4949.57& 	-7741.34& 	-1927.50& 	-4065.83& 	-6743.16\\
  4 & 100 &-6094.91& 	-10166.83& 	-18501.75& 	-4537.32& 	-8533.31& 	-16319.39\\
   & 200 &-8143.13& 	-18400.38& 	-31427.90& 	-3281.08& 	-14422.20& 	-28218.68\\
  \bottomrule
  \end{tabular*}
  \end{center}
\end{table}

\begin{table}[H]
  \begin{center}
  \caption{Computational time of BB and Fmincon: convex sub-problem (Second)}
  \renewcommand\arraystretch{0.8}
  \begin{tabular*}{\textwidth}{@{\extracolsep{\fill}}ccd{4.5}d{4.5}d{4.5}d{4.5}d{4.5}d{3.5}@{\extracolsep{\fill}}}
  \toprule
  ~ & ~ & \multicolumn{3}{c}{BB}&  \multicolumn{3}{c}{Fmincon}\\
  \cmidrule{3-5}\cmidrule{6-8}
  \multicolumn{1}{c}{$K$} &  \multicolumn{1}{c}{$n$} &    \multicolumn{1}{c}{Max} &    \multicolumn{1}{c}{Ave} &   \multicolumn{1}{c}{Min}     & \multicolumn{1}{c}{Max}  &  \multicolumn{1}{c}{Ave} &   \multicolumn{1}{c}{Min} \\
  \midrule
   & 50 & 1.10 &	0.96 &	0.83 & 	0.70& 	0.57& 	0.45\\
  2 & 100 &3.05 &	2.77 &	2.53 & 	2.02& 	1.67& 	1.44\\
   & 200 &15.08 &	13.34 &	10.58 & 	9.40& 	8.18& 	5.83\\
  \midrule
   & 50 &7.59 &	4.08 &	2.04 & 	0.84& 	0.54& 	0.36\\
  3 & 100 &17.80 &	9.88 &	5.35 & 	6.24& 	2.73& 	1.54\\
   & 200 &77.37 &	58.30 &	23.77 & 	12.46& 	11.66& 	10.09\\
  \midrule
   & 50 &12.87 &	7.56 &	3.44 & 	0.69& 	0.62& 	0.48\\
  4 & 100 &86.15 &	40.69 &	9.58  &	2.82& 	2.01& 	1.65\\
   & 200 &297.62 &	184.20 &	100.57 & 	15.72& 	12.66& 	10.42\\
  \bottomrule
  \end{tabular*}
  \end{center}
\end{table}

Then we execute other experiments with the same parameter setting except for $\pi_i$, $i=1,\cdots,K$, which are randomly generated such that there is at least one $\pi_i$ satisfies $\pi_i<2\alpha$. In this case, the sub-problems of the BB algorithm are non-convex. Table 3 and Table 4 display the optimal value and the corresponding computational time. We can see that the computational time of BB algorithm is much larger than that when sub-problems are convex, even though the problem size is smaller.

\begin{table}[H]
  \begin{center}
  \caption{Optimal value obtained by BB and Fmincon: non-convex sub-problem}
  \renewcommand\arraystretch{0.8}
  \begin{tabular*}{\textwidth}{@{\extracolsep{\fill}}ccd{5.4}d{5.4}d{5.4}d{5.4}d{5.4}d{5.4}@{\extracolsep{\fill}}}
  \toprule
  ~ & ~ & \multicolumn{3}{c}{BB} & \multicolumn{3}{c}{Fmincon}\\
  \cmidrule{3-5}\cmidrule{6-8}
    \multicolumn{1}{c}{$K$} &  \multicolumn{1}{c}{$n$} &    \multicolumn{1}{c}{Max} &    \multicolumn{1}{c}{Ave} &   \multicolumn{1}{c}{Min}     & \multicolumn{1}{c}{Max}  &  \multicolumn{1}{c}{Ave} &   \multicolumn{1}{c}{Min} \\
  \midrule
   & 10 & 19.01& 	-797.56& 	-2441.26& 	20.66& 	 -784.15& 	-2441.26\\
  2 & 20 &-876.37& 	-3886.61& 	-5533.25& 	-845.59& -3871.81& 	-5533.25\\
   & 30 &-42.20& 	-2987.99& 	-6297.43& 	-9.35& 	 -2825.91& 	-6253.98\\
  \midrule
   & 10 & -923.75& 	-1560.27& 	-2265.11& 	-789.21& 	-1498.22& 	-2265.11\\
  3 & 20 &-3602.49& 	-4117.10& 	-4488.56& 	-3231.04& 	-3979.91& 	-4488.56\\
   & 30 &-1816.84& 	-3128.75& 	-5495.17& 	-1666.50& 	-2911.59& 	-5407.16\\
  \midrule
   & 10 &-1016.62& 	-1468.16& 	-1760.37& 	-961.03& 	-1428.68& 	-1739.06\\
  4 & 20 &-47.44& 	    -1981.41& 	-4281.85& 	-25.34& 	-1884.08& 	-4177.93\\
   & 30 &-1691.45& 	-5069.70& 	-6711.06& 	-1356.68& 	-4767.04& 	-6610.53\\
  \bottomrule
  \end{tabular*}
  \end{center}
\end{table}

\begin{table}[H]
  \begin{center}
  \caption{Computational time of BB and Fmincon: non-convex sub-problem (Second)}
  \renewcommand\arraystretch{0.8}
  \begin{tabular*}{\textwidth}{@{\extracolsep{\fill}}ccd{5.4}d{5.4}d{4.5}d{4.5}d{4.5}d{3.5}@{\extracolsep{\fill}}}
  \toprule
  ~ & ~ & \multicolumn{3}{c}{BB} & \multicolumn{3}{c}{Fmincon}\\
  \cmidrule{3-5}\cmidrule{6-8}
    \multicolumn{1}{c}{$K$} &  \multicolumn{1}{c}{$n$} &    \multicolumn{1}{c}{Max} &    \multicolumn{1}{c}{Ave} &   \multicolumn{1}{c}{Min}     & \multicolumn{1}{c}{Max}  &  \multicolumn{1}{c}{Ave} &   \multicolumn{1}{c}{Min} \\
  \midrule
   & 10 &250.94 &	156.82 &	8.88 & 	0.33& 	0.17& 	0.06\\
  2 & 20 &563.30 &157.49 &0.43 & 	0.20& 	0.16& 	0.10\\
   & 30 &838.02 &	556.79 &	363.75 &	3.28& 	0.92& 	0.32\\
  \midrule
   & 10 &75.74 &	33.84 &	0.29 & 	0.10& 	0.08& 	0.06\\
  3 & 20 &1411.34 &	542.89 &	19.10 & 	0.14& 	0.14& 	0.13\\
   & 30 &8656.94 &	3030.72 &	724.30 &	0.74& 	0.39& 	0.30\\
  \midrule
   & 10 &568.24 &	196.93 &	0.34 & 	0.20& 	0.13& 	0.10\\
  4 & 20 &1047.51 &	445.32 &	156.69 &	1.42& 	0.60& 	0.26\\
   & 30 &27044.59 &	8830.70 &	748.77 &	0.48& 	0.43& 	0.38\\
 \bottomrule
  \end{tabular*}
  \end{center}
\end{table}

\section{Conclusion}
In this paper, we study the chance constrained program with quadratic randomness, where random variables that constitute the quadratic randomness are supposed to follow GMD. The key finding for solving this problem is that under some mild conditions, the asymptotic distribution of the quadratic randomness is a univariate GMD, which is employed to approximate the distribution of the quadratic randomness and thus an effective BB method can be used to find the global solution. Furthermore, due to the fitting capability of GMD, the proposed method is a unified approach that could be applied in many different situations.  

An interesting finding of this paper is that using condition number constrained GMD in the real application can not only reduce the asymptotic approximation error, but also learn the characteristics of the data very well. However, the cost of this constrained GMD is that more components are needed, which leads to more computational times for the BB algorithm. Although we have proved the density function of condition number constrained GMD can approximate any density function to any precision, it is still unclear how many components are exactly needed given the approximation precision. In addition, the algorithm to efficiently estimate the parameters of the constrained GMD is also needed to be investigated. We believe that these issues are worth exploring in the future.

\section*{Appendix}

\subsection*{Appendix A: Proof of Proposition \ref{moment}}
We first introduce a relevant lemma from \cite{zhu2020}.
\begin{lemma}\label{zhu}
Given the decision vector $\bm{x}$,
\begin{eqnarray*}
E_{p_j}(c(\bm{\xi},\bm{x}))&=&\frac{1}{2}{\rm tr}(A\Sigma_j)+\frac{1}{2}\bm{\mu}_j^{\top}A\bm{\mu}_j+\bm{a}^{\top}\bm{\mu}_j+a
\\
V_{p_j}(c(\bm{\xi},\bm{x}))&=&\frac{1}{2}{\rm tr}((A\Sigma_j)^2)+(A\bm{\mu}_j+\bm{a})^{\top}\Sigma_j(A\bm{\mu}_j+\bm{a})
\end{eqnarray*}
respectively. Here, ${\rm tr}(\cdot)$ is the trace of a matrix.
\end{lemma}
\begin{proof}
For then mean, we have
\begin{eqnarray*}
E_{p^*}(c(\bm{\xi},\bm{x}))=\sum_{i=1}^K\pi_iE_{p_i}(c(\bm{\xi},\bm{x})).
\end{eqnarray*}

For the variance, we have
\begin{eqnarray*}
V_{p^*}(c(\bm{\xi},\bm{x}))&=&E_{p^*}(c^2(\bm{\xi},\bm{x}))-E_{p^*}^2(c(\bm{\xi},\bm{x}))\\
&=&\sum_{i=1}^{K}\pi_iE_{p_i}(c^2(\bm{\xi},\bm{x}))-E_{p^*}^2(c(\bm{\xi},\bm{x}))\\
&=&\sum_{i=1}^{K}\pi_i(E_{p_i}^2(c(\bm{\xi},\bm{x}))+V_{p_i}(c(\bm{\xi},\bm{x})))-E_{p^*}^2(c(\bm{\xi},\bm{x})).
\end{eqnarray*}
Notice that
\begin{eqnarray*}
&&\sum_{i=1}^{K}\pi_iE_{p_i}^2(c(\bm{\xi},\bm{x}))-E_{p^*}^2(c(\bm{\xi},\bm{x}))\\
&&=\sum_{i=1}^{K}\pi_iE_{p_i}^2(c(\bm{\xi},\bm{x}))-(\sum_{i=1}^{K}\pi_iE_{p_i}(c(\bm{\xi},\bm{x})))^2\\
&&=\sum_{i=1}^K\pi_i(1-\pi_i)E_{p_i}^2(c(\bm{\xi},\bm{x}))-\sum_{i\neq j}\pi_i\pi_jE_{p_i}(c(\bm{\xi},\bm{x}))E_{p_j}(c(\bm{\xi},\bm{x}))\\
&&=\sum_{i=1}^K\pi_i(\sum_{j\neq i}\pi_j)E_{p_i}^2(c(\bm{\xi},\bm{x}))-\sum_{i\neq j}\pi_i\pi_jE_{p_i}(c(\bm{\xi},\bm{x}))E_{p_j}(c(\bm{\xi},\bm{x}))\\
&&=\sum_{1\leq i<j\leq K}\pi_i\pi_j(E_{p_i}(c(\bm{\xi},\bm{x}))-E_{p_j}(c(\bm{\xi},\bm{x})))^2.
\end{eqnarray*}
Therefore, we have
\begin{eqnarray*}
V_{p^*}(c(\bm{\xi},\bm{x})) &=& \sum_{i=1}^{K}V_{p_i}(c(\bm{\xi},\bm{x}))+\sum_{1\leq i<j\leq K}\pi_i\pi_j(E_{p_i}(c(\bm{\xi},\bm{x}))-E_{p_j}(c(\bm{\xi},\bm{x})))^2.
\end{eqnarray*}
The proof is completed.
\end{proof}

\subsection*{Appendix B: L\'{e}vy's continuity Lemma}
\begin{lemma}
For a sequence of random variables $\{X_n\}_{n=1}^{+\infty}$ and $X$, $X_n\stackrel{d}{\rightarrow}X$ if and only if $\lim\limits_{n\rightarrow+\infty}E\left(e^{{\rm i}tX_n}\right)=E\left(e^{{\rm i}tX}\right)$ for every $t\in\mathbb{R}$. Here, $\stackrel{d}{\rightarrow}$ means convergence in distribution.
\end{lemma}

\subsection*{Appendix C: Introduction of a Lemma for Theorem \ref{fitness}}
Before the introduction of the lemma, we recall some notations. Since the conclusion of this lemma holds for any $\mathscr{I}_i$, $i\in\{1,\cdots,N\}$, without loss of generality, we rewrite $\Sigma_i$, $\mathscr{I}_i$ and $\underline{f}_i$ as $\Sigma$, $\mathscr{I}$ and $\underline{f}$. The corresponding notations are 
\begin{eqnarray*} 
p(\bm{z})&=&p(\bm{z}|\bm{0},\Sigma)=\frac{1}{(2\pi)^{\frac{m}{2}}|\Sigma|^{-\frac{1}{2}}}e^{-\frac{1}{2}\bm{z}^{\top}\Sigma^{-1}\bm{z}},\\
\chi_{\mathscr{I}}(\bm{z})&=&\underline{f}\mathbf{1}_{\mathscr{I}}(\bm{z})=\left\{\begin{array}{cc}\underline{f}&\bm{z}\in \mathscr{I}\\0&{\rm otherwise}\\ \end{array}\right.,\\
\chi_{\mathscr{I}}\ast p(\bm{z})&=&\int_{\mathbb{R}^m}\chi_{\mathscr{I}}(\bm{y})p(\bm{z}-\bm{y})d\bm{y}=\underline{f}\int_{\mathscr{I}}p(\bm{z}-\bm{y})d\bm{y}.
\end{eqnarray*}

Let $\mathscr{I}_1,\cdots,\mathscr{I}_n$ be a partition of $\mathscr{I}$, that is, any entry of $\mathscr{I}$ belongs to one and only one of $\mathscr{I}_j$, $j=1,\cdots,n$. Denote a function $\chi_{\mathscr{I}}\ast p_n(\bm{z})=\sum_{j=1}^n\underline{f}|\mathscr{I}_j|p(\bm{z}-\bm{y}_j)$, where $\bm{y}_j\in\mathscr{I}_j$, $j=1,\cdots,n$. Denote $|\mathscr{I}_{\max}|=\max\limits_{j\in\{1,\cdots,n\}}\{|\mathscr{I}_j|\}$. By the definition of Riemann integral, the relation between $\chi_{\mathscr{I}}\ast p(\bm{z})$ and $\chi_{\mathscr{I}}\ast p_n(\bm{z})$ is that for any given $\bm{z}$,
\begin{eqnarray*}
\lim_{|\mathscr{I}_{\max}|\rightarrow0}\chi_{\mathscr{I}}\ast p_n(\bm{z})=\chi_{\mathscr{I}}\ast p(\bm{z}).
\end{eqnarray*}

The following lemma indicates that this convergence is a uniform convergence.
\begin{lemma}\label{fitness_lemma}
If $\mathscr{B}$ is a compact set, then for any given $\epsilon>0$, there exists a partition $\mathscr{I}_1,\cdots,\mathscr{I}_n$ of $\mathscr{I}$ such that
\begin{eqnarray*}
&&\int_{\mathscr{B}}|\chi_{\mathscr{I}}\ast p(\bm{z})-\chi_{\mathscr{I}}\ast p_n(\bm{z})|d\bm{z}<\epsilon.
\end{eqnarray*}
\end{lemma}
\begin{proof}
First we show that for any given $\epsilon>0$, there exists a $\delta$ that is only related to $\epsilon$ such that if $\bm{y}_1,\bm{y}_2\in\mathbb{R}^m$, $|\bm{y}_1-\bm{y}_2|_{\infty}<\delta$, where $|\cdot|_{\infty}$ represents the maximal absolute value among the entries of a vector, then $|p(\bm{z}-\bm{y}_1)-p(\bm{z}-\bm{y}_2)|<\frac{\epsilon}{\underline{f}|\mathscr{I}||\mathscr{B}|}$. 

If $\bm{z}$ is given, $p(\bm{z}-\bm{y})=p(\bm{y}-\bm{z})$ is a Gaussian density function with mean vector $\bm{z}$, which is uniformly continuous, so by definition there exists a $\delta(\bm{z})$ such that if $\bm{y}_1,\bm{y}_2\in\mathbb{R}^m$, $|\bm{y}_1-\bm{y}_2|_{\infty}<\delta(\bm{z})$, then 
\begin{eqnarray*}
&&|p(\bm{z}-\bm{y}_1)-p(\bm{z}-\bm{y}_2)|<\frac{\epsilon}{\underline{f}|\mathscr{I}||\mathscr{B}|}. 
\end{eqnarray*}

Consider another point $\bm{z}'$ and let $\Delta\bm{z}=\bm{z}'-\bm{z}$. If $|\bm{y}_1-\bm{y}_2|_{\infty}<\delta(\bm{z})$, then 
\begin{eqnarray*}
&&|(\bm{y}_1-\Delta\bm{z})-(\bm{y}_2-\Delta\bm{z})|_{\infty}<\delta(\bm{z})
\end{eqnarray*}
and then
\begin{eqnarray*}
&&|p(\bm{z}'-\bm{y}_1)-p(\bm{z}'-\bm{y}_2)|=|p(\bm{z}-(\bm{y}_1-\Delta\bm{z}))-p(\bm{z}-(\bm{y}_2-\Delta\bm{z}))|<\frac{\epsilon}{\underline{f}|\mathscr{I}||\mathscr{B}|},
\end{eqnarray*}
which implies that given $\epsilon$, there exists a $\delta$ that is invariant with respect to $\bm{z}$ and $\bm{y}$ such that if $\bm{y}_1,\bm{y}_2\in\mathbb{R}^m$, $|\bm{y}_1-\bm{y}_2|_{\infty}<\delta$, 
\begin{eqnarray*}
&&|p(\bm{z}-\bm{y}_1)-p(\bm{z}-\bm{y}_2)|<\frac{\epsilon}{\underline{f}|\mathscr{I}||\mathscr{B}|}. 
\end{eqnarray*}

Then let the partition $\mathscr{I}_1,\cdots,\mathscr{I}_n$ satisfy that for any $\bm{y}_1,\bm{y}_2\in\mathscr{I}_j$, $j=1,\cdots,n$, $|\bm{y}_1-\bm{y}_2|_{\infty}<\delta$. Then we have 
\begin{eqnarray*}
|\chi_{\mathscr{I}}\ast p(\bm{z})-\chi_{\mathscr{I}}\ast p_n(\bm{z})|&=&\left|\underline{f}\sum_{j=1}^{n}\int_{\mathscr{I}_j}(p(\bm{z}-\bm{y})-p(\bm{z}-\bm{y}_j))d\bm{y}\right|\\
&\leq&\underline{f}\sum_{j=1}^n\int_{\mathscr{I}_j}\left|p(\bm{z}-\bm{y})-p(\bm{z}-\bm{y}_j)\right|d\bm{y}\\
&\leq&\underline{f}\sum_{j=1}^n\frac{|\mathscr{I}_j|\epsilon}{\underline{f}|\mathscr{I}||\mathscr{B}|}=\frac{\epsilon}{|\mathscr{B}|}.
\end{eqnarray*}

Finally, we have
\begin{eqnarray*}
&&\int_{\mathscr{B}}|\chi_{\mathscr{I}}\ast p(\bm{z})-\chi_{\mathscr{I}}\ast p_n(\bm{z})|d\bm{z}<\epsilon.
\end{eqnarray*}
The proof is completed.
\end{proof}

\subsection*{Appendix D: Branch-and-bound algorithm proposed by \cite{hu2021}}
\renewcommand{\algorithmicrequire}{\textbf{Input:}}
\renewcommand{\algorithmicensure}{\textbf{Output:}}
{\small
\begin{algorithm}[H]
\caption{branch-and-bound (BB)}
\begin{algorithmic}[1]
\REQUIRE $i=0$, $v^*=+\infty$, $\mathscr{Y}^0=\{\bm{y}\in\mathbb{R}^K:\underline{y}_i\leq y_i\leq\overline{y}_i,i=1,\cdots,K\}$.
\ENSURE Global optimal solution $\bm{x}^*$ and global optimal value $v^*$.
\STATE Solve problem ${\rm QCCP_{relax}}$.\\
\textbf{if} $\rm QCCP_{relax}$ is infeasible, the problem $\rm QCCP_{lin}$ is infeasible, stop.\\
\textbf{else}, $v_0:=$ the optimal value, $(\bm{x}^c,\bm{y}^c):=$ the optimal solution.\\
~~~~\textbf{if} $\bm{x}^c$ is feasible to $\rm QCCP_{lin}$, $v^*:=v_0$, $\bm{x}^*:=\bm{x}^c$, stop.\\
~~~~\textbf{else}, $\Omega=\{(\mathscr{Y}^i,v_i)\}$. Set $i:=i+1$. Go to step \ref{step_2}.\\
~~~~\textbf{end if}\\
\textbf{end if}
\label{step_1}
\STATE \textbf{if} $\Omega=\emptyset$, $\bm{x}^*$ is the optimal solution. Stop.\\
\textbf{else} Go to step \ref{step_3}.\\
\textbf{end if}
\label{step_2}
\STATE Chose and remove one branch $(\mathscr{Y}^i,v_i)$ from $\Omega$ with minimal $v_i$. Subdivide the branch $\mathscr{Y}^i$ into $\mathscr{Y}^i_1$ and $\mathscr{Y}^i_2$ by dividing the longest edge of $\mathscr{Y}^i$ at its midpoint. Denote by $\mathscr{Y}^i_1$ the upper rectangle that includes the upper corner $(\overline{y}_1,\cdots,\overline{y}_K)$ and $\mathscr{Y}^i_2$ the lower rectangle that includes the upper corner $(\underline{y}_1,\cdots,\underline{y}_K)$.\\
\textbf{if} $\sum_{j=1}^K\pi_j\underline{y}^1_j>1-\alpha$, where $(\underline{y}^1_1,\cdots,\underline{y}^1_K)$ is the lower corner of $\mathscr{Y}^i_1$, set $\underline{v}^i_2=v_i$ and $\Omega:=\Omega\bigcup(\mathscr{Y}^i_2,\underline{v}^i_2)$. Set $i:=i+1$. Go to step \ref{step_2}.\\
\textbf{else if} $\sum_{j=1}^K\pi_j\overline{y}^2_j<1-\alpha$, where $(\overline{y}^2_1,\cdots,\overline{y}^2_K)$ is the upper corner of $\mathscr{Y}^i_2$, solve problem $\rm{CCP_{relax}}$ with $\mathscr{Y}_1$.\\
~~~~\textbf{is} it is infeasible, $\Omega:=\Omega$, $i:=i+1$, go to step \ref{step_2}.\\
~~~~\textbf{else} denote the optimal value $\underline{v}^i_1$ and the optimal solution $(\bm{x}^i_1,\bm{y}^i_1)$\\
~~~~~~~~\textbf{if} $\underline{v}^i_1>v^*$, $\Omega:=\Omega$, $i:=i+1$, go to step \ref{step_2}.\\
~~~~~~~~\textbf{else if} $\bm{x}^i_1$ is feasible to $\rm QCCP_{lin}$, $v^*:=\underline{v}^i_1$, $\bm{x}^*:=\bm{x}^i_1$, and delete from $\Omega$ all branches $(\mathscr{Y}_i,v_i)$ with $v_i\geq v^*$. Set $i:=i+1$, go to step \ref{step_2}.\\
~~~~~~~~\textbf{else} $\Omega:=\Omega\bigcup(\mathscr{Y}^i_1,\underline{v}^i_1)$, $i:=i+1$, go to step \ref{step_2}.\\
~~~~~~~~\textbf{end if}\\
~~~~\textbf{end if}\\
\textbf{else} let $\underline{v}^i_2=v_i$. Set $\Omega=\Omega\bigcup(\mathscr{Y}_2,\underline{v}^i_2)$. Solve problem $\rm{QCCP_{relax}}$ with $\mathscr{Y}_1$ and derive the optimal value $\underline{v}^i_1$.\\
~~~~\textbf{if} it is infeasible, $i:=i+1$, go to step \ref{step_2}.\\
~~~~\textbf{else} denote the optimal value $\underline{v}^i_1$ and the optimal solution $(\bm{x}^i_1,\bm{y}^i_1)$.\\
~~~~~~~~\textbf{if} $\underline{v}^i_1>v^*$, $i:=i+1$, go to step \ref{step_2}.\\
~~~~~~~~\textbf{else if} $\bm{x}^i_1$ is feasible to $\rm QCCP_{lin}$, $v^*:=\underline{v}^i_1$, $\bm{x}^*:=\bm{x}^i_1$, and delete from $\Omega$ all instances $(\mathscr{Y}_i,v_i)$ with $v_i\geq v^*$. Set $i:=i+1$, go to step \ref{step_2}.\\
~~~~~~~~\textbf{else}, $\Omega:=\Omega\bigcup(\mathscr{Y}^i_1,\underline{v}^i_1)$, $i:=i+1$, go to step \ref{step_2}.\\
~~~~~~~~\textbf{end if}\\
~~~~\textbf{end if}\\
\textbf{end if}
\label{step_3}
\RETURN $\bm{x}^*$, $v^*$.
\end{algorithmic}
\end{algorithm}
}

\bibliographystyle{ormsv080}
\bibliography{mixture}

\begin{thebibliography}{30}
\expandafter\ifx\csname natexlab\endcsname\relax\def\natexlab#1{#1}\fi
\expandafter\ifx\csname url\endcsname\relax
  \def\url#1{{\tt #1}}\fi
\expandafter\ifx\csname urlprefix\endcsname\relax\def\urlprefix{URL }\fi
\expandafter\ifx\csname urlstyle\endcsname\relax
  \expandafter\ifx\csname doi\endcsname\relax
  \def\doi#1{doi:\discretionary{}{}{}#1}\fi \else
  \expandafter\ifx\csname doi\endcsname\relax
  \def\doi{doi:\discretionary{}{}{}\begingroup \urlstyle{rm}\Url}\fi \fi

\bibitem[{Calafiore and Campi(2005)}]{campi2005}
Calafiore, G., M.~C. Campi. 2005.
\newblock Uncertain convex programs: Randomized solutions and confidence
  levels.
\newblock {\it Mathematical Programming\/} {\bf 102}(1) 25--46.

\bibitem[{Calafiore and El~Ghaoui(2006)}]{calafiore2006}
Calafiore, G.~C., L.~El~Ghaoui. 2006.
\newblock On distributionally robust chance-constrained linear programs.
\newblock {\it Journal of Optimization Theory and Applications\/} {\bf 130}(1)
  1--22.

\bibitem[{Campi and Garatti(2008)}]{campi2008}
Campi, M.~C., S.~Garatti. 2008.
\newblock The exact feasibility of randomized solutions of uncertain convex
  programs.
\newblock {\it SIAM Journal on Optimization\/} {\bf 19} 1211--1230.

\bibitem[{Campi and Garatti(2011)}]{campi2011}
Campi, M.~C., S.~Garatti. 2011.
\newblock A sampling-and-discarding approach to chance-constrained
  optimization: Feasibility and optimality.
\newblock {\it Journal of Optimization Theory and Applications\/} {\bf 148}
  257--280.

\bibitem[{Charnes et~al.(1958)Charnes, Cooper, and Symonds}]{charnes1958}
Charnes, A., W.~W. Cooper, G.~H. Symonds. 1958.
\newblock Cost horizons and certainty equivalents: An approach to stochastic
  programming of heating oil.
\newblock {\it Management Science\/} {\bf 4} 235--263.

\bibitem[{Chen et~al.(2022)Chen, Zhang, and Song}]{chen2022}
Chen, G., H.~C. Zhang, Y.~H. Song. 2022.
\newblock {Chance-constrained DC optimal power flow with non-gaussian
  distributed uncertainties}.
\newblock {\it 2022 IEEE Power $\&$ Energy Society General Meeting\/}. 1--5.
\newblock \doi{10.1109/PESGM48719.2022.9916658}.

\bibitem[{Chen et~al.(2018)Chen, Peng, and Liu}]{chen2018}
Chen, Z.~P., S.~Peng, J~Liu. 2018.
\newblock Data-driven robust chance constrained problems: A mixture model
  approach.
\newblock {\it Journal of Optimization Theory and Applications\/} {\bf 179}
  1065--1085.

\bibitem[{Cui et~al.(2013)Cui, Zhu, Sun, and Li}]{cui2013}
Cui, X.~T., S.~S. Zhu, X.~L. Sun, D.~Li. 2013.
\newblock Nonlinear portfolio selection using approximating parametric
  value-at-risk.
\newblock {\it Journal of Banking and Finance\/} {\bf 37}(6) 2124--2139.

\bibitem[{El~Ghaoui et~al.(2003)El~Ghaoui, Oks, and Oustry}]{el2003}
El~Ghaoui, L., M.~Oks, F.~Oustry. 2003.
\newblock Worst-case value-at-risk and robust portfolio optimization: A conic
  programming approach.
\newblock {\it Operations Research\/} {\bf 51}(4) 543--556.

\bibitem[{Golub and Van~Loan(2013)}]{golub2013}
Golub, G.~H., C.~F. Van~Loan. 2013.
\newblock {\it Matrix Computations\/}.
\newblock 4th ed. The Johns Hopkins University, Baltimore.

\bibitem[{Hanasusanto et~al.(2017)Hanasusanto, Roitch, Kuhn, and
  Wiesemann}]{hanasusanto2017}
Hanasusanto, G.~A., V.~Roitch, D.~Kuhn, W.~Wiesemann. 2017.
\newblock Ambiguous joint chance constraints under mean and dispersion
  information.
\newblock {\it Operations Research\/} {\bf 65}(3) 751--767.

\bibitem[{Henrion and Moller(2012)}]{henrion2012}
Henrion, R., A.~Moller. 2012.
\newblock A gradient formula for linear chance constraints under gaussian
  distribution.
\newblock {\it Mathematics of Operations Research\/} {\bf 37}(3) 475--488.

\bibitem[{Hu et~al.(2022)Hu, Sun, and Zhu}]{hu2021}
Hu, Z.~L., W.~J. Sun, S.~S. Zhu. 2022.
\newblock Chance constrained programs with gaussian mixture models.
\newblock {\it IISE Transactions\/} {\bf 54}(12) 1117--1130.

\bibitem[{Hull(2009)}]{hull2009}
Hull, J. 2009.
\newblock {\it Options, Futures and Other Derivatives\/}.
\newblock Pearson Prentice Hall, New Jersey.

\bibitem[{Jorion(2007)}]{jorion2007}
Jorion, P. 2007.
\newblock {\it Value at Risk: The New Benchmark for Managing Financial Risk\/}.
\newblock McGraw-Hill Companies, New York.

\bibitem[{Kishida and Nagahara(2023)}]{kishida2023}
Kishida, M., M.~Nagahara. 2023.
\newblock Risk-aware maximum hands-off control using worst-case conditional
  value-at-risk.
\newblock {\it IEEE Transactions on Automatic Control\/}
  \doi{10.1109/TAC.2023.3235246}.

\bibitem[{Luedtke and Ahmed(2008)}]{luedtke2008}
Luedtke, J., S.~Ahmed. 2008.
\newblock A sample approximation approach for optimization with probabilistic
  constraints.
\newblock {\it SIAM Journal on Optimization\/} {\bf 19} 674--699.

\bibitem[{Marron and Wand(1992)}]{marron1992}
Marron, J.~S., M.~P. Wand. 1992.
\newblock Exact mean integrated squared error.
\newblock {\it Annals of Statistics\/} {\bf 20} 712--736.

\bibitem[{McLachlan and Peel(2000)}]{mclachlan}
McLachlan, G., D~Peel. 2000.
\newblock {\it Finite Mixture Models\/}.
\newblock Wiley, New York.

\bibitem[{McLachlan and Krishnan(1997)}]{mclachlan1997}
McLachlan, G.~J., T.~Krishnan. 1997.
\newblock {\it The EM algorithm and Extensions\/}.
\newblock John Wiley \& Sons, New York.

\bibitem[{Miller and Wagner(1965)}]{miller1965}
Miller, L.~B., H.~Wagner. 1965.
\newblock Chance-constrained programming with joint constraints.
\newblock {\it Operations Research\/} {\bf 13} 930--945.

\bibitem[{Nemirovski and Shapiro(2006)}]{nemirovski2006}
Nemirovski, A., A.~Shapiro. 2006.
\newblock Convex approximation of chance constrained programs.
\newblock {\it SIAM Journal on Optimization\/} {\bf 17}(4) 969--996.

\bibitem[{Pearson(1894)}]{pearson1894}
Pearson, K. 1894.
\newblock Contributions to the mathematical theory of evolution.
\newblock {\it Philosophical Transactions of the Royal Society of London\/}
  {\bf 54} 326--330.

\bibitem[{Pr\'ekopa(1970)}]{prekopa1970}
Pr\'ekopa, A. 1970.
\newblock On probabilistic programming.
\newblock In proceedings of the Princeton Symposium on Mathematical
  Programming, Princeton University Press, Princeton, NJ, 113--138.

\bibitem[{Ren et~al.(2022)Ren, Ahn, and Kamgarpour}]{ren2022}
Ren, K., H.~Ahn, M.~Kamgarpour. 2022.
\newblock Chance-constrained trajectory planning with multimodal environmental
  uncertainty.
\newblock {\it IEEE Control Systems Letters\/} {\bf 7} 13--18.

\bibitem[{Shapiro et~al.(2009)Shapiro, Dentcheva, and
  Ruszczynski}]{shapiro2009}
Shapiro, A., D.~Dentcheva, A.~Ruszczynski. 2009.
\newblock {\it Lectures on Stochastic Programming: Modeling and Theory\/}.
\newblock Society for Industrial and Applied Mathematics, Philadelphia, PA.

\bibitem[{Van~der Vaart(1998)}]{van1998}
Van~der Vaart, A.~W. 1998.
\newblock {\it Asymptotic Statistics\/}.
\newblock Cambridge University Press.

\bibitem[{Wilson(2000)}]{wilson2000}
Wilson, R. 2000.
\newblock {MGMM: multiresolution Gaussian mixture models for computer vision}.
\newblock {\it Proceedings 15th International Conference on Pattern
  Recognition. ICPR-2000\/}, vol.~1. 212--215.
\newblock \doi{10.1109/ICPR.2000.905305}.

\bibitem[{Zhu et~al.(2020)Zhu, Zhu, Pei, and Cui}]{zhu2020}
Zhu, S.~S., W.~Zhu, X.~Pei, X.~T. Cui. 2020.
\newblock Hedging crash risk in optimal portfolio selection.
\newblock {\it Journal of Banking and Finance\/} {\bf 119} 105905.

\bibitem[{Zymler et~al.(2013)Zymler, Kuhn, and Rustem}]{zymler2013}
Zymler, S., D.~Kuhn, B.~Rustem. 2013.
\newblock Distributionally robust joint chance constraints with second-order
  moment information.
\newblock {\it Mathematical Programming\/} {\bf 137} 167--198.

\end{thebibliography}
\end{document}